\documentclass[11pt]{article}
\usepackage{amsfonts}
\usepackage{amsmath}
\usepackage{amssymb}
\usepackage{graphicx}
\usepackage{bbold}
\usepackage{upgreek}
\usepackage[colorlinks]{hyperref}
\usepackage{palatino}
\usepackage[usenames,dvipsnames]{xcolor}
\usepackage{palatino}
\usepackage{makeidx}

\hypersetup{colorlinks=true, linkcolor=blue, citecolor=green,
filecolor=black, urlcolor=black }
\providecommand{\U}[1]{\protect\rule{.1in}{.1in}}
\providecommand{\U}[1]{\protect\rule{.1in}{.1in}}
\newtheorem{theorem}{Theorem}

\newtheorem{corollary}[theorem]{Corollary}

\newtheorem{definition}[theorem]{Definition}
\newtheorem{example}[theorem]{Example}

\newtheorem{lemma}[theorem]{Lemma}
\newtheorem{notation}[theorem]{Notation}

\newtheorem{proposition}[theorem]{Proposition}
\newtheorem{remark}[theorem]{Remark}

\newenvironment{proof}[1][Proof]{\noindent\textbf{#1.} }{\ \rule{0.5em}{0.5em}}

\renewcommand{\thefootnote}{\fnsymbol{footnote}}
\begin{document}

\title{Stochastic Variational Inequalities on Non--Convex Domains}
\author{Rainer Buckdahn$^{a}$, Lucian Maticiuc$^{b,c}$, \'{E}tienne Pardoux$^{d}$ and Aurel R\u{a}\c{s}canu$^{b,e}$
$\bigskip$\\{\small $^{a}$ Laboratoire de Math\'{e}matiques, CNRS--UMR 6205, Universit\'{e} de Bretagne Occidentale,}\smallskip\\
{\small 6, avenue Victor Le Gorgeu, 29238 Brest cedex 3, France.}\medskip\\
{\small $^{b}$~Faculty of Mathematics, \textquotedblleft Alexandru Ioan Cuza\textquotedblright\ University,}\smallskip\\
{\small Carol I Blvd., no. 11, Iasi, 700506, Romania.}\medskip\\
{\small $^{c}$~Department of Mathematics, \textquotedblleft Gheorghe Asachi\textquotedblright\ Technical University of Ia\c{s}i,}\smallskip\\
{\small Carol I Blvd., no. 11, Iasi, 700506, Romania.}\medskip\\
{\small $^{d}$~LATP, Universit\'{e} de Provence, CMI, 39, rue F. Joliot--Curie, 13453, Marseille Cedex 13, France}\medskip\\
{\small $^{e}$~\textquotedblleft Octav Mayer\textquotedblright\ Mathematics Institute of the Romanian Academy, Iasi branch,}\smallskip\\
{\small Carol I Blvd., no. 8, Iasi, 700506, Romania.}}
\date{}
\maketitle

\begin{abstract}
The objective of this work is to prove in a first step the existence
and the uniqueness of a solution of the follo\-wing multivalued
deterministic differential equation:
\begin{equation*}
\left\{
\begin{array}{l}
dx(t)+\partial^{-}\varphi(x(t))(dt)\ni dm(t),\quad t>0,\smallskip\\
x(0)=x_{0}~,
\end{array}
\right.
\end{equation*}
where $m:\mathbb{R}_{+}\rightarrow\mathbb{R}^{d}$ is a continuous
function and $\partial^{-}\varphi$ is the Fr\'{e}chet
subdifferential of a $(\rho ,\gamma)$--semiconvex function $\varphi$; the
domain of $\varphi$ can be non--convex, but some
regularities of the boundary are required.

The continuity of the map $m\mapsto
x:C([0,T];\mathbb{R}^{d})\rightarrow C([0,T] ;\mathbb{R}^{d})$
associating to the input function $m$ the solution $x$ of the
above equation, as well as tightness criteria allow to pass from the
above deterministic case to the following stochastic variational
inequality driven by a multi--dimensional Brownian motion:

\begin{equation*}
\left\{
\begin{array}{l}
X_{t}+K_{t}=\xi+{\displaystyle\int_{0}^{t}}F(s,X_{s})ds+{\displaystyle\int_{0}^{t}}G(s,X_{s})
dB_{s},\quad t\geq0,\smallskip \\
dK_{t}(\omega)\in\partial^{-}\varphi( X_{t}(\omega))(dt).
\end{array}
\right.
\end{equation*}
\end{abstract}

\textbf{AMS Classification subjects:} 60H10, 60J60.\medskip

\textbf{Keywords or phrases:} Skorohod Problem; Stochastic Variational
Inequalities; Fr\'{e}chet Subdifferential. \renewcommand{\thefootnote}{%
\fnsymbol{footnote}} \footnotetext{\textit{\scriptsize E--mail addresses:}
{\scriptsize rainer.buckdahn@univ-brest.fr (Rainer Buckdahn),
lucian.maticiuc@uaic.ro (Lucian Maticiuc), pardoux@cmi.univ-mrs.fr (\'{E}%
tienne Pardoux), aurel.rascanu@uaic.ro (Aurel R\u{a}\c{s}canu)}}

\section{Introduction}

Given a multi--dimensional Brownian motion $B=(B_{t})$ and a proper $(\rho ,\gamma)$--semiconvex function $\varphi $ (for the definition
the reader is referred to Definition \ref{semiconvex function} from the next
section) defined over a possibly non--convex domain Dom$(\varphi )$ and a
random variable $\xi $ independent of $B$, which takes its values in the
closure of \textrm{Dom}$(\varphi )$, we are interested in the following
multi--valued stochastic differential equations (also called stochastic
variational inequality) driven by the Fr\'{e}chet subdifferential operator $%
\partial ^{-}\varphi $:
\begin{equation}
\left\{
\begin{array}{l}
X_{t}+K_{t}=\xi +\displaystyle\int_{0}^{t}F\left( s,X_{s}\right)
ds+\int_{0}^{t}G\left( s,X_{s}\right) dB_{s},\;t\geq 0,\smallskip \\
dK_{t}\left( \omega \right) \in \partial ^{-}\varphi \left( X_{t}\left(
\omega \right) \right) \left( dt\right) .%
\end{array}%
\right.  \label{GSSE intr}
\end{equation}%
However, in order to study the above system, we shall first solve the
following deterministic counterpart of the above equation. Given a
continuous function $m:\mathbb{R}_{+}\rightarrow \mathbb{R}^{d}$ and an
initial value $x_{0}\in \overline{\mathrm{Dom}(\varphi )}$ we look for a
pair of continuous functions $\left( x,k\right) :\mathbb{R}_{+}\rightarrow
\mathbb{R}^{2d}$ such that

\begin{equation}
\left\{
\begin{array}{l}
x\left( t\right) +k\left( t\right) =x_{0}+\displaystyle\int_{0}^{t}f\left(
s,x\left( s\right) \right) ds+m\left( t\right) ,\quad t\geq 0,\smallskip \\
dk\left( t\right) \in \partial ^{-}\varphi \left( x\left( t\right) \right)
\left( dt\right)%
\end{array}%
\right.  \label{GSE intr}
\end{equation}%
(for the notation $dk\left( t\right) \in \partial ^{-}\varphi \left( x\left(
t\right) \right) \left( dt\right) $ see Definition \ref{semiconvex function}
and Remark \ref{notation_apartenance}).

With the particular choice of $f\equiv 0$ and $\varphi $ as convexity
indicator function of a closed domain $E\subset \mathbb{R}^{d}$,%
\begin{equation*}
\varphi \left( x\right) =I_{E}\left( x\right) :=\left\{
\begin{array}{cc}
0, & \text{if }x\in E,\smallskip \\
+\infty , & \text{if }x\notin E,%
\end{array}%
\right.
\end{equation*}%
equation (\ref{GSE intr}) turns turns out to be just the Skorohod problem,
i.e., a reflection problem, associated to the data $x_{0},m$ and $E$. For
this reason we will refer to equation (\ref{GSSE intr}) as Skorohod equation.

The existence of solutions for both the Skorohod equation (\ref{GSE intr})
and for the stochastic equation (\ref{GSSE intr}), has been well studied by
different authors for the case, where $\varphi $ is a convex function. In
this case $\partial ^{-}\varphi $ becomes a maximal monotone operator and
the domain Dom$(\varphi )$ in which the solution is kept is convex. By
replacing $\partial ^{-}\varphi $ by a general maximal monotone operator $A,$
E. C\'{e}pa generalized in \cite{ce/98} the above equation in the finite
dimensional case, while A. R\u{a}\c{s}canu \cite{ra/96} investigated the
infinite dimensional case.

Deterministic variational inequalities with regular inputs, i.e.,
deterministic equations of type (\ref{GSE intr}), with convex $\varphi $ and
$m=0$ have been well studied and the corresponding results have by now
become classical. As concerns the non--convex framework, the reader is
referred to A. Marino, M. Tosques \cite{ma-to/90}, M. Degiovanni, A. Marino,
M. Tosques \cite{de-ma-to/85}, A. Marino, C. Saccon, M. Tosques \cite%
{ma-sa-to/89} or R. Rossi, G. Savar\'{e} \cite{ro-sa/04}, \cite{ro-sa/06}.
They used the concept of $\phi $--convexity (see, e.g., \cite[Definition 4.1]%
{ma-to/90}) and they provide the existence, uniqueness and continuous
dependence on the initial data of the solutions for the evolution equation
of type (\ref{GSE intr}) in the case $m=0$, $f=0$ (or $f\left( t,x\left(
t\right) \right) =f\left( t\right) $ in \cite{ro-sa/06}). We precise that
our notion of $\left( \rho ,\gamma \right) $--semiconvex function
corresponds to the particular case of a $\phi $--convex function of order $%
r=1$ with $\phi \left( x,y,\varphi \left( x\right) ,\varphi \left( y\right)
,\alpha \right) :=\rho +\gamma \left\vert \alpha \right\vert .$ This
particular form was required by the presence of the singular input $dm/dt.$

Related to our problem is the research on non--convex differential
inclusions, see, e.g., F. Papalini \cite{pa/95}, T. Cardinali, G. Colombo,
F. Papalini, M. Tosques \cite{ca-co-pa/97} and A. Cernea, V. Staicu \cite%
{ce-st/02}. In \cite{ce-st/02} it is proved the existence of a solution for
the Cauchy problem%
\begin{equation*}
x^{\prime }\in -\partial ^{-}\varphi \left( x\right) +F\left( x\right)
+f\left( t,x\right) ,\quad x\left( 0\right) =x_{0},
\end{equation*}%
where $\varphi $ is a $\phi $--convex function of order $r=2$, $F$ is a
upper semicontinuous and cyclically monotone multifunction and $f$ is a
Carath\'{e}odory function.$\smallskip $

The particular case of a reflection problem, i.e., with $\varphi =I_{E}$,
was extended to that of moving domains $E\left( t\right) ,\ t\geq 0$, by
considering the following problem (which solution is called \textit{sweeping
process}):
\begin{equation*}
\left\{
\begin{array}{l}
x^{\prime }\left( t\right) +N_{E\left( t\right) }\left( x\left( t\right)
\right) \ni f\left( t,x\left( t\right) \right) ,\quad t\geq 0,\medskip \\
x\left( 0\right) =x_{0},%
\end{array}%
\right.
\end{equation*}%
where $N_{E\left( t\right) }\left( x\left( t\right) \right) =\partial
^{-}I_{E\left( t\right) }\left( x\left( t\right) \right) $ is the external
normal cone to $E\left( t\right) $ in $x\left( t\right) $.

This problem was introduced by J.J. Moreau (see \cite{mo/77}) for the case $%
f\equiv 0$ and convex sets $E\left( t\right) ,\ t\geq 0$, and it has been
intensively studied since then by several authors; see, e.g., C. Castaing
\cite{ca/83}, M.D.P. Monteiro Marques \cite{mo/93}. For the case of a
sweeping process without the assumption of convexity on the sets $E\left(
t\right) ,\ t\geq 0$, we refer to \cite{be/00}, \cite{be-ca/96}, \cite%
{co-go/99} and \cite{co-mo/03}. The extension to the case $f\not\equiv 0$
was considered, e.g., in C. Castaing, M.D.P. Monteiro Marques \cite{ca-mo/96}
and J.F. Edmond, and L. Thibault \cite{ed-th/06} (See also the references
therein). Another extension was made in \cite{mi-ro/07} and \cite{ro/12} by
considering the quasivariational sweeping process%
\begin{equation*}
\left\{
\begin{array}{l}
x^{\prime }\left( t\right) +N_{E\left( t,x\left( t\right) \right) }\left(
x\left( t\right) \right) \ni 0,\quad t\geq 0,\medskip \\
x\left( 0\right) =x_{0}.%
\end{array}%
\right.
\end{equation*}%
The works mentioned above concern the case with vanishing driving force $m=0$%
. Let us discuss now the case of equation (\ref{GSE intr}) with singular
input $dm/dt$. The associated reflection problem with singular input $dm/dt$
has been investigated by A.V. Skorohod in \cite{sk/61} and \cite{sk/62} (for
the particular case of $E=[0,\infty )$), by H. Tanaka in \cite{ta/79} (for a
general convex domain $E$), and by P--L. Lions and A.S. Sznitman in \cite%
{li-sn/84} and Y. Saisho in \cite{sa/87} for a non--convex domains.
Generalizations from the reflection problem (with $\partial \varphi
=\partial I_{E}$) to the case of $\partial \varphi $ for a general convex
function $\varphi $, and even to the case of a maximal monotone operator $A$%
, were discussed by A. R\u{a}\c{s}canu in \cite{ra/96}, V. Barbu and A. R%
\u{a}\c{s}canu in \cite{ba-ra/97} and by E. C\'{e}pa in \cite{ce/95} and
\cite{ce/98}.

Concerning the stochastic equations, we shall mention the papers \cite%
{li-sn/84} by P--L. Lions and A.S. Sznitman and \cite{sa/87} by Y. Saisho,
but also \cite{du-is/93} by P. Dupuis and H. Ishii for stochastic
differential equations with reflecting boundary conditions. On the other
hand, A. R\u{a}\c{s}canu \cite{ra/96}, I. Asiminoaei and A. R\u{a}\c{s}canu
\cite{as-ra/97} as well as A. Bensoussan and A. R\u{a}\c{s}canu \cite%
{be-ra/97} studied stochastic variational inequalities (\ref{GSSE intr}) in
the convex case.

More recently, A. M. Gassous, A. R\u{a}\c{s}canu and E. Rotenstein obtained
in \cite{ga-ra-ro/12} existence and uniqueness results for stochastic
variational inequalities with oblique subgradients. More precisely, in their
equation the direction of reflection at the boundary of the convex domain
differs from the normal direction, an effect which is caused by the presence
of a multiplicative Lipschitz matrix acting on the subdifferential operator.
In the authors' approach it turned out to be crucial to pass first by a
study of the Skorohod problem with generalized reflection.

The objective of the present work is twice: We generalize both the
(non--)convex reflection problem as well as convex variational inequalities
to non--convex variational inequalities. Some studies in this direction have
been made already by A. R\u{a}\c{s}canu and E. Rotenstein in \cite{ra-ro/13}%
: a non--convex setup for multivalued (deterministic) differential equations
driven by oblique subgradients has been established and the uniqueness and
the local existence of the solution has been proven.

Our approach here in the present manuscript is heavily based on an a priori
discussion of the generalized Skorohod problem (\ref{GSE intr}) with $%
f\equiv 0$ and a $\left( \rho ,\gamma \right) $--semiconvex $\varphi $. We
give useful a priori estimates and prove the existence and the uniqueness of
a solution $\left( x,k\right) $ for the generalized Skorohod problem:
\begin{equation}
\left\{
\begin{array}{l}
x\left( t\right) +k\left( t\right) =x_{0}+m\left( t\right) ,\quad t\geq
0,\smallskip \\
dk\left( t\right) \in \partial ^{-}\varphi \left( x\left( t\right) \right)
\left( dt\right) ,%
\end{array}%
\right.  \label{GSE intr 2}
\end{equation}%
where $x_{0}\in \overline{\mathrm{Dom}\left( \varphi \right) }$, the input $%
m $ is a continuous function starting from zero, and $\partial ^{-}\varphi $
is the Fr\'{e}chet subdifferential of a proper, lower semicontinuous and $%
\left( \rho ,\gamma \right) $--semiconvex function $\varphi $. Here the set
\textrm{Dom}$\left( \varphi \right) $ is not necessarily convex, but however
two assumptions are required:\medskip

$1.$ \textrm{Dom}$\left( \varphi \right) $ satisfies the \textit{uniform
exterior ball condition} (see Definition \ref{UEBC});

$2.$ $\mathrm{Dom}\left( \varphi \right) \;$ satisfies the $\left( \gamma
,\delta ,\sigma \right) $--\textit{shifted uniform interior ball condition},
i.e.\medskip

\noindent there are some suitable constants $\gamma \geq 0$ and $\delta
,\sigma >0$ such that, for all $y\in \mathrm{D}$\textrm{om}$\left( \varphi
\right) $, there are some $\lambda _{y}\in (0,1]$\ and $v_{y}\in \mathbb{R}%
^{d}$, $\left\vert v_{y}\right\vert \leq 1$ with $\lambda _{y}-\left(
\left\vert v_{y}\right\vert +\lambda _{y}\right) ^{2}\gamma \geq \sigma $
and
\begin{equation*}
\overline{B}\left( x+v_{y},\lambda _{y}\right) \subset \mathrm{Dom}\left(
\varphi \right) ,\quad \mbox{for all
}~x\in \mathrm{Dom}\left( \varphi \right) \cap \overline{B}\left( y,\delta
\right) .
\end{equation*}%
We observe that this condition is fulfilled if, in particular, the domain
\textrm{Dom}$\left( \varphi \right) $ satisfies the \textit{uniform interior
drop condition} (see Definition \ref{drop cond}). It is worth pointing out
that the shifted uniform interior ball condition is comparable with
assumption $\left( 5\right) $ of P--L. Lions, A.S. Sznitman \cite{li-sn/84}
(or \textit{Condition (B)} from \cite{sa/87}) (see the Remarks \ref{remark 2}%
--\ref{remark 3}).\medskip

The application $m\mapsto x:C\left( \left[ 0,T\right] ;\mathbb{R}^{d}\right)
\rightarrow C\left( \left[ 0,T\right] ;\mathbb{R}^{d}\right) $, which
associates to the input function $m$ the solution $x$ of (\ref{GSE intr 2}),
will be proven to be continuous. This allows to derive the existence of a
solution to the associated stochastic equation with additive noise $%
M=(M_{t}) $:
\begin{equation}
\left\{
\begin{array}{l}
X_{t}\left( \omega \right) +K_{t}\left( \omega \right) =\xi \left( \omega
\right) +M_{t}\left( \omega \right) ,\;t\geq 0,\;\omega \in \Omega
,\smallskip \\
dK_{t}\left( \omega \right) \in \partial ^{-}\varphi \left( X_{t}\left(
\omega \right) \right) \left( dt\right) .%
\end{array}%
\right.  \label{GSSE intr 2}
\end{equation}%
After having the existence, the uniqueness and properties for the equations (%
\ref{GSE intr 2}) and (\ref{GSSE intr 2}), we will be able to extend the
study to the more general equations (\ref{GSE intr}) and (\ref{GSSE intr}),
where $f$ (respectively $F$) is continuous, satisfies a one-sided Lipschitz
condition with respect to the second variable and a boundedness
assumption.\medskip

The article is organized as it follows: The next section is devoted to a
recall of such basic notions as those of a as semiconvex set, a $\left( \rho
,\gamma \right) $--semiconvex function or a Fr\'{e}chet subdifferential.
Some notions, like for instance that of a $\left( \rho ,\gamma \right) $%
--semiconvex function, are illustrated by an example. In Section 3 we give
the definition of the solution to the generalized Skorohod problem (\ref{GSE
intr 2}), we prove the existence and the uniqueness of a solution $\left(
x,k\right) $ and we give some useful a priori estimations. Moreover, we
extend equation (\ref{GSE intr 2}) to the stochastic case (\ref{GSSE intr 2}%
). Section 4 is devoted to the proof of the both main results of Section 3.
Finally, the Sections 5 and 6 study the extension of the results established
in Section 3 to the equations (\ref{GSE intr}) and (\ref{GSSE intr}). The
Appendix is devoted to important auxiliary results such as applications of
Fatou's Lemma, some complements concerning tightness in $C\left( \mathbb{R}%
_{+};\mathbb{R}^{d}\right) $ or a very useful forward stochastic inequality,
which are used in our approach.

\section{Preliminaries}

We introduce first some definitions and results concerning the notions of
normal cone, uniform exterior ball conditions, semiconvex sets, $\left( \rho
,\gamma \right) $--semiconvex functions and Fr\'{e}chet subdifferential of a
function.

Here and everywhere below $E$ will be a nonempty closed subset of $\mathbb{R}%
^{d}$. Let $N_{E}\left( x\right) $ be the closed external normal cone of $E$
at $x\in\mathrm{Bd}\left( E\right) $ i.e.%
\begin{equation*}
N_{E}\left( x\right) :=\left\{ u\in\mathbb{R}^{d}:\lim_{\delta\searrow 0}%
\frac{\mathrm{d}_{E}\left( x+\delta u\right) }{\delta}=\left\vert
u\right\vert \right\} ,
\end{equation*}
where $\mathrm{d}_{E}\left( z\right) :=\inf\left\{ \left\vert z-x\right\vert
:x\in E\right\} $ is the distance of a point $z\in \mathbb{R}^{d}$ to $E$.

\begin{definition}
\label{UEBC}Let $r_{0}>0$. We say that $E$ satisfies the $r_{0}$--uniform
exterior ball condition (we write it $r_{0}-UEBC$ for brevity) if%
\begin{equation*}
B\left( x+u,r_{0}\right) \cap E=\emptyset ,\quad \text{where }B\left(
x,r\right) \text{ denotes the ball from }\mathbb{R}^{d}\text{ of centre }x%
\text{ and radius }r
\end{equation*}%
or equivalently if%
\begin{equation*}
N_{E}\left( x\right) \neq \left\{ 0\right\} \quad \text{for all }x\in
\mathrm{Bd}\left( E\right)
\end{equation*}%
and%
\begin{equation*}
\text{for all }x\in \mathrm{Bd}\left( E\right) \text{ and }u\in N_{E}\left(
x\right) \text{ such that }\left\vert u\right\vert =r_{0}\text{ it holds
that }\mathrm{d}_{E}\left( x+u\right) =r_{0}\,.
\end{equation*}
\end{definition}

We remark that for all $v\in N_{E}\left( x\right) $ with $\left\vert
v\right\vert \leq r_{0}$ we also have $\mathrm{d}_{E}\left( x+v\right)
=\left\vert v\right\vert .$

\begin{definition}
\label{semiconvex set}Let $\gamma\geq0.$ A set $E$ is $\gamma$--semiconvex
if for all $x\in\mathrm{Bd}\left( E\right) $ there exists $\hat{x}\in
\mathbb{R}^{d}\mathbb{\smallsetminus}\left\{ 0\right\} $ such that%
\begin{equation*}
\left\langle \hat{x},y-x\right\rangle \leq\gamma\left\vert \hat{x}%
\right\vert \left\vert y-x\right\vert ^{2},\quad\forall y\in E.
\end{equation*}
\end{definition}

We have the following equivalence:

\begin{lemma}[{see \protect\cite[Lemma 6.47]{pa-ra/14}}]
\label{UEBC - semiconvex}Let $r_{0}>0$. The set $E$ satisfies the $%
r_{0}-UEBC $ if and only if $E$ is $\tfrac{1}{2r_{0}}$--semiconvex.
\end{lemma}

For a given $z\in \mathbb{R}^{d}$ we denote by $\Pi _{E}\left( z\right) $
the set of elements $x\in E$ such that $\mathrm{d}_{E}\left( z\right)
=\left\vert z-x\right\vert $. Obviously, $\Pi _{E}\left( z\right) $ is non
empty since $E$ is non empty and closed. Moreover, under the $r_{0}$%
--uniform exterior ball condition, it follows that the set $\Pi _{E}\left(
z\right) $ is a singleton for all $z$ such that $\mathrm{d}_{E}\left(
z\right) <r_{0}$. In this case $\pi _{E}\left( z\right) $ will denote the
unique element of $\Pi _{E}\left( z\right) $ and it is called the projection
of $z$ on $E$. We recall the following well--known property of the
projection.

\begin{lemma}
\label{normal cone}Let the $r_{0}-UEBC$ be satisfied, $\varepsilon \in
\left( 0,r_{0}\right) $ and $\overline{U}_{\varepsilon }\left( E\right)
:=\left\{ y\in \mathbb{R}^{d}:\mathrm{d}_{E}\left( y\right) \leq \varepsilon
\right\} $ denoting the closed $\varepsilon $--neighborhood of $E$.

Then:

\begin{itemize}
\item the closed external normal cone of $E$ at $x$ is given by%
\begin{equation*}
N_{E}\left( x\right) =\left\{ \hat{x}:\left\langle \hat{x},y-x\right\rangle
\leq\frac{1}{2r_{0}}\left\vert \hat{x}\right\vert \left\vert y-x\right\vert
^{2},\;\forall y\in E\right\} ;
\end{equation*}

\item the projection $\pi_{E}$ restricted to $\overline{U}_{\varepsilon
}\left( E\right) $ is Lipschitz with Lipschitz constant $L_{\varepsilon
}=r_{0}/\left( r_{0}-\varepsilon\right) $;

and

\item the function $\mathrm{d}_{E}^{2}\left( \cdot\right) $ is of class $%
C^{1}$ on $\overline{U}_{\varepsilon}\left( E\right) $ with%
\begin{equation*}
\frac{1}{2}\nabla\mathrm{d}_{E}^{2}\left( z\right) =z-\pi_{E}\left( z\right)
\quad\text{and}\quad z-\pi_{E}\left( z\right) \in N_{E}\left( \pi_{E}\left(
z\right) \right) ,\;\forall z\in\overline{U}_{\varepsilon }\left( E\right) .
\end{equation*}
\end{itemize}
\end{lemma}

Let us introduce now the notion of \textit{drop of vertex }$x$\textit{\ and
running direction }$v$.

Let $x,v\in \mathbb{R}^{d}$, $r>0.$ The set%
\begin{equation*}
D_{x}\left( v,r\right) :=\mathrm{conv}\left\{ x,\overline{B}\left(
x+v,r\right) \right\} =\left\{ x+t\left( u-x\right) :u\in \overline{B}\left(
x+v,r\right) ,\;t\in \left[ 0,1\right] \right\}
\end{equation*}%
is called $\left( v,r\right) $--drop of vertex $x$ and running direction $v$%
. Remark that if $\left\vert v\right\vert \leq r,$ then $D_{x}\left(
v,r\right) =\overline{B}\left( x+v,r\right) .$

\begin{definition}
\label{drop cond}The set $E\subset\mathbb{R}^{d}$ satisfies the uniform
interior drop condition if there exist $r_{0},h_{0}>0$ such that for all $%
x\in E$ there exists $v_{x}\in\mathbb{R}^{d}$ with $\left\vert
v_{x}\right\vert \leq h_{0}$ and%
\begin{equation*}
D_{x}\left( v_{x},r_{0}\right) \subset E
\end{equation*}
(we also say that $E$ satisfies the uniform interior $\left(
h_{0},r_{0}\right) $--drop condition).
\end{definition}

\begin{remark}
\label{suff cond for drop}It is easy to see that if there exists $r_{0}>0$
such that $E^{c}$ satisfies the the $r_{0}-UEBC$, then $E$ satisfies the
uniform interior $\left( h_{0},r_{0}\right) $--drop condition.$\smallskip $

Indeed let $x\in Bd\left( E^{c}\right) =Bd\left( E\right) $ and $u_{x}\in
N_{E^{c}}\left( x\right) $ with $\left\vert u_{x}\right\vert =r_{0}$.

Then%
\begin{equation*}
D_{x}\left( u_{x},r_{0}/2\right) \subset D_{x}\left( u_{x},r_{0}\right) =%
\overline{B}\left( x+u_{x},r_{0}\right) \subset E.
\end{equation*}
It is easy to see that, for any $x\in int\left( E\right) $, there exists a
direction $v_{x}\in\mathbb{R}^{d}$ such that $\left\vert v_{x}\right\vert
\leq r_{0}$ and $D_{x}\left( v_{x},r_{0}/2\right) \subset E.$
\end{remark}

We state below that the drop condition implies a weaker condition, but is
not equivalent with this (for the proof see Proposition 4.35 in \cite%
{pa-ra/14}).

\begin{proposition}
\label{cond for SUIBC}Let the set $E$ be as above with $\mathrm{Int}\left(
E\right) \neq \emptyset $. If set $E$ satisfies the uniform interior $\left(
h_{0},r_{0}\right) $--drop condition then $E$ \textit{satisfies the shifted
uniform interior ball condition, which means that there exist }$\gamma \geq
0 $ and $\delta ,\sigma >0$\textit{, and for every }$y\in E$\textit{\ there
exist }$\lambda _{y}\in (0,1]$\textit{\ and }$v_{y}\in \mathbb{R}^{d},$%
\textit{\ }$\left\vert v_{y}\right\vert \leq 1$\textit{\ such that}%
\begin{equation}
\begin{array}{rl}
\left( i\right) & \lambda _{y}-\left( \left\vert v_{y}\right\vert +\lambda
_{y}\right) ^{2}\gamma \geq \sigma ,\smallskip \\
\left( ii\right) & \overline{B}\left( x+v_{y},\lambda _{y}\right) \subset
E,\quad \forall ~x\in E\cap \overline{B}\left( y,\delta \right)%
\end{array}
\label{SUIBC cond}
\end{equation}%
(this condition will be called $\left( \gamma ,\delta ,\sigma \right) $--%
\textit{SUIBC}).
\end{proposition}

\begin{example}
\label{ex for SUIBC}Let $E$ be a set for which there exists a function $\phi
\in C_{b}^{2}\left( \mathbb{R}^{d}\right) $ such that%
\begin{equation*}
\begin{array}{rl}
\left( i\right) & E=\left\{ x\in \mathbb{R}^{d}:\phi \left( x\right) \leq
0\right\} ,\smallskip \\
\left( ii\right) & \mathrm{Int}\left( E\right) =\left\{ x\in \mathbb{R}%
^{d}:\phi \left( x\right) <0\right\} ,\smallskip \\
\left( iii\right) & \mathrm{Bd}\left( E\right) =\left\{ x\in \mathbb{R}%
^{d}:\phi \left( x\right) =0\right\} \;\text{and }\left\vert \nabla \phi
\left( x\right) \right\vert =1,\;\forall ~x\in \mathrm{Bd}\left( E\right) .%
\end{array}%
\end{equation*}%
Then the set $E$ satisfies the uniform exterior ball condition and the
uniform interior drop condition.$\smallskip $

Indeed, using the definition of $E$ we see that, for $x\in\mathrm{Bd}\left(
E\right) ,$ the gradient $\nabla\phi\left( x\right) $ is a unit normal
vector to the boundary, pointing towards the exterior of $E$. Therefore, for
any $x\in\mathrm{Bd}\left( E\right) $, the normal cone is given by $%
N_{E}\left( x\right) =\left\{ c\nabla\phi\left( x\right) :c\geq 0\right\} $
and $N_{E^{c}}\left( x\right) =\left\{ -c\nabla\phi\left( x\right)
:c\geq0\right\} .$

Since $\phi \left( y\right) \leq 0=\phi \left( x\right) $, for all $y\in E$,
$x\in \mathrm{Bd}\left( E\right) ,$%
\begin{align*}
\left\langle \nabla \phi \left( x\right) ,y-x\right\rangle & =\phi \left(
y\right) -\phi \left( x\right) -\int_{0}^{1}\left\langle \nabla \phi \left(
x+\lambda \left( y-x\right) \right) -\nabla \phi \left( x\right)
,y-x\right\rangle d\lambda \\
& \leq M\left\vert y-x\right\vert ^{2}=M\left\vert \nabla \phi \left(
x\right) \right\vert \left\vert y-x\right\vert ^{2},
\end{align*}%
which means, using Definition \ref{semiconvex set} and Lemma \ref{UEBC -
semiconvex}, that $E$ satisfies $\frac{1}{2M}-UEBC$.

Since $\phi \left( y\right) \geq 0=\phi \left( x\right) $, for all $y\in
E^{c}$, $x\in \mathrm{Bd}\left( E\right) ,$%
\begin{align*}
\left\langle -\nabla \phi \left( x\right) ,y-x\right\rangle & =-\phi \left(
y\right) +\int_{0}^{1}\left\langle \nabla \phi \left( x+\lambda \left(
y-x\right) \right) -\nabla \phi \left( x\right) ,y-x\right\rangle d\lambda \\
& \leq M\left\vert y-x\right\vert ^{2}=M\left\vert -\nabla \phi \left(
x\right) \right\vert \left\vert y-x\right\vert ^{2},
\end{align*}%
which yields that $E^{c}$ satisfies $\frac{1}{2M}-UEBC$ and consequently
(see Remark \ref{suff cond for drop}) $E$ satisfies the uniform interior
drop condition.
\end{example}

If $E$ denotes a closed subset of $\mathbb{R}^{d}$ let $E_{\varepsilon}$ be
the $\varepsilon$--interior of $E$, i.e.%
\begin{equation*}
E_{\varepsilon}:=\left\{ x\in E:\mathrm{d}_{E^{c}}\left( x\right)
\geq\varepsilon\right\} .
\end{equation*}

\begin{example}
Let $E\subset \mathbb{R}^{d}$ be a closed convex set with nonempty interior.
If there exists $r_{0}>0$ such that$\;$(the $r_{0}$--interior of $E$) $%
E_{r_{0}}\neq \emptyset $ and $h_{0}=\sup_{z\in E}\mathrm{d}%
_{E_{r_{0}}}\left( z\right) <\infty $ (in particular if $E$ is a bounded
closed convex set with nonempty interior), then $E$ satisfies the uniform
interior $\left( h_{0},r_{0}\right) $--drop condition.

Moreover for every $0<\delta\leq\frac{r_{0}}{2\left( 1+h_{0}\right) }\wedge1$%
, $E$ satisfies $\left( \gamma,\delta,\sigma\right) $--\textit{SUIBC }with $%
\lambda_{y}=\sigma=\delta.\smallskip$

For the proof, let $y\in E,$ $\hat{y}$ the projection of $y$ on the set $%
E_{r_{0}}$ and $v_{y}=\frac{1}{1+h_{0}}\left( \hat{y}-y\right) $. Hence $%
\left\vert \hat{y}-y\right\vert \leq h_{0}$, $\left\vert v_{y}\right\vert
\leq1$ and for all $x\in E\cap\overline{B}\left( y,\delta\right) $%
\begin{equation*}
\overline{B}\left( x+v_{y},\delta\right) \subset\overline{B}\Big(y+v_{y},%
\frac{r_{0}}{1+h_{0}}\Big)\subset\mathrm{conv}\left\{ y,\overline{B}\left(
\hat{y},r_{0}\right) \right\} =D_{y}\left( \hat{y}-y,r_{0}\right) \subset E.
\end{equation*}
$\smallskip$
\end{example}

Let $\varphi :$ $\mathbb{R}^{d}\rightarrow (-\infty ,+\infty ]$ be a
function with domain defined by%
\begin{equation*}
\mathrm{Dom}\left( \varphi \right) :=\left\{ v\in \mathbb{R}^{d}:\varphi
\left( v\right) <+\infty \right\} .
\end{equation*}%
We recall now the definition of the Fr\'{e}chet subdifferential (for this
kind of subdifferential operator see, e.g., \cite{ma-to/90} and the
monograph \cite{ro-we/97}, cap. VIII):

\begin{definition}
The Fr\'{e}chet subdifferential $\partial^{-}\varphi$ is defined by:

$\mathrm{a}_{\mathrm{1}}\mathrm{)}\;\partial^{-}\varphi\left( x\right)
=\emptyset,\;$if $x\notin\mathrm{Dom}\left( \varphi\right) \mathrm{\ }$and

$\mathrm{a}_{\mathrm{2}}\mathrm{)}\;$for $x\in\mathrm{Dom}\left(
\varphi\right) ,$%
\begin{equation*}
\partial^{-}\varphi\left( x\right) =\{\hat{x}\in\mathbb{R}^{d}:\underset{%
y\rightarrow x}{\lim\inf}\dfrac{\varphi\left( y\right) -\varphi\left(
x\right) -\left\langle \hat{x},y-x\right\rangle }{\left\vert y-x\right\vert }%
\geq0\}.
\end{equation*}
\end{definition}

Taking into account this definition we will say that $\varphi $ is proper if
the domain $\mathrm{Dom}\left( \varphi \right) \neq \emptyset $ and has no
isolated points.

We set%
\begin{equation*}
\begin{array}{l}
\mathrm{Dom}\left( \partial ^{-}\varphi \right) =\left\{ x\in \mathbb{R}%
^{d}:\partial ^{-}\varphi \left( x\right) \neq \emptyset \right\} ,\smallskip
\\
\partial ^{-}\varphi =\left\{ \left( x,\hat{x}\right) :\;x\in \mathrm{Dom}%
\left( \partial ^{-}\varphi \right) ,\;\hat{x}\in \partial ^{-}\varphi
\left( x\right) \right\} .%
\end{array}%
\end{equation*}%
In the particular case of the indicator function of the closed set $E$,%
\begin{equation*}
\varphi \left( x\right) =I_{E}\left( x\right) :=\left\{
\begin{array}{cc}
0, & \text{if }x\in E,\smallskip \\
+\infty , & \text{if }x\notin E,%
\end{array}%
\right.
\end{equation*}%
the function $\varphi $ is lower semicontinuous and the Fr\'{e}chet
subdifferential becomes%
\begin{equation*}
\partial ^{-}I_{E}\left( x\right) =\Big\{\hat{x}\in \mathbb{R}%
^{d}:\limsup_{y\rightarrow x,~y\in E}\dfrac{\left\langle \hat{x}%
,y-x\right\rangle }{\left\vert y-x\right\vert }\leq 0\Big\}.
\end{equation*}%
Moreover, in this particular case we deduce that, for any closed subset $E$
of a Hilbert space,%
\begin{equation}
\partial ^{-}I_{E}\left( x\right) =N_{E}\left( x\right)
\label{particular Frechet subdifferential}
\end{equation}%
(for the proof see Colombo and Goncharov \cite{co-go/01}).

\begin{definition}
\label{semiconvex function}Let $\rho,\gamma\geq0$. The function $\varphi :%
\mathbb{R}^{d}\rightarrow(-\infty,+\infty]$ is said to be $\left(
\rho,\gamma\right) $--semiconvex function if$\smallskip$

$\mathrm{a}_{\mathrm{1}}\mathrm{)\;}\overline{\mathrm{Dom}\left(
\varphi\right) }$ is $\gamma$--semiconvex,$\smallskip$

$\mathrm{a}_{\mathrm{2}}\mathrm{)\;Dom}\left( \partial^{-}\varphi\right)
\neq\emptyset,\smallskip$

$\mathrm{a}_{\mathrm{3}}\mathrm{)\;}$For all $\left( x,\hat{x}\right) \in
\partial ^{-}\varphi $ and $y\in \mathbb{R}^{d}$:%
\begin{equation*}
\left\langle \hat{x},y-x\right\rangle +\varphi \left( x\right) \leq \varphi
\left( y\right) +\left( \rho +\gamma \left\vert \hat{x}\right\vert \right)
\left\vert y-x\right\vert ^{2}.
\end{equation*}
\end{definition}

\begin{remark}
Let $E$ be a nonempty closed subset of $\mathbb{R}^{d}.$ We have:

\begin{enumerate}
\item $I_{E}$ is $\left( 0,\gamma \right) $--semiconvex iff $E$ is $\gamma $%
--semiconvex (see (\ref{particular Frechet subdifferential}) and Definitions %
\ref{semiconvex set} and \ref{semiconvex function}; we also mention that in
the definition of $\gamma $--semiconvexity we can take $x\in E$, but in this
case $\hat{x}$ should be taken $0$).

\item $I_{E}$ is $\left( 0,\gamma \right) $--semiconvex (or, see \cite[%
Definition 4.1]{ma-to/90}, $\phi $--convex of order $r=1$ with $\phi
(x,y,\varphi \left( x\right) $, $\varphi \left( y\right) ,\alpha )=\gamma
\left\vert \alpha \right\vert $) iff there exists $\delta ,\mu >0$ such that
$x\mapsto d_{E}\left( x\right) +\mu \left\vert x\right\vert ^{2}$ is convex
on $B\left( y,\delta \right) $, for any $y\in E$ (see \cite[Lemma 6.47]%
{pa-ra/14}).

\item A convex function is also $\left( \rho ,\gamma \right) $--semiconvex
function, for any $\rho ,\gamma \geq 0$ (see Definition \ref{semiconvex
function} and the definition of the subdifferential of a convex function).

\item If $E$ is convex, then $E$ is $\gamma $--semiconvex for any $\gamma
\geq 0$ (see the supporting hyperplane Theorem 4.1.6 from \cite{bo-le/06}).

\item If $E$ has nonempty interior and is $0$--semiconvex, then $E$ is
convex (see Definition \ref{semiconvex set} and \cite[Exercise 2.27]%
{bo-va/04}).

\item If $\varphi :\mathbb{R}^{d}\rightarrow (-\infty ,+\infty ]$ is a $%
\left( \rho ,\gamma \right) $--semiconvex function, then there exists $%
a,b\in \mathbb{R}_{+}$ and $c\in \mathbb{R}$ such that%
\begin{equation*}
\varphi \left( y\right) +a\left\vert y\right\vert ^{2}+b\left\vert
y\right\vert ^{2}+c\geq 0,\quad \text{for all }y\in \mathbb{R}^{d}.
\end{equation*}%
Indeed, by Definition \ref{semiconvex function}, we have, for a fixed $%
\left( x_{0},\hat{x}_{0}\right) \in \partial ^{-}\varphi :a=\rho +\gamma
\left\vert \hat{x}_{0}\right\vert $, $b=2a\left\vert x_{0}\right\vert
+\left\vert \hat{x}_{0}\right\vert $ and $c=a\left\vert x_{0}\right\vert
^{2}+\left\langle \hat{x}_{0},x_{0}\right\rangle -\varphi \left(
x_{0}\right) $.
\end{enumerate}
\end{remark}

\begin{example}
\label{example_semiconvex}If the bounded set $E$ satisfy the $r_{0}$--$UEBC$
and $g\in C^{1}\left( \mathbb{R}^{d}\right) $ (or $g:\mathbb{R}%
^{d}\rightarrow \mathbb{R}$ is a convex function), then $\varphi :\mathbb{R}%
^{d}\rightarrow (-\infty ,+\infty ]$ given by%
\begin{equation*}
\varphi \left( x\right) :=I_{E}\left( x\right) +g\left( x\right)
\end{equation*}%
is a lower semicontinuous and $\big(\frac{L}{2r_{0}},\frac{1}{2r_{0}}\big)$%
--semiconvex function, where $\left\vert \nabla g\left( x\right) \right\vert
\leq L$, for any $x\in E$ (or $\left\vert \partial g\left( x\right)
\right\vert \leq L$, for any $x\in E$). Moreover%
\begin{equation*}
\left\vert \varphi \left( x\right) -\varphi \left( y\right) \right\vert \leq
L\left\vert x-y\right\vert ,\;\forall \;x,y\in \mathrm{Dom}\left( \varphi
\right) =E.
\end{equation*}
\end{example}

In order to define the solution for the deterministic problem envisaged by
our work it is necessary to introduce the bounded variation function space.

Let $T>0$, $k:\left[ 0,T\right] \rightarrow\mathbb{R}^{d}$ and $\mathcal{D}$
be the set of the partitions of the interval $\left[ 0,T\right] $.

Set%
\begin{equation*}
S_{\Delta}(k)=\sum\limits_{i=0}^{n-1}|k(t_{i+1})-k(t_{i})|
\end{equation*}
and%
\begin{equation}
\left\updownarrow k\right\updownarrow _{T}:=\sup\limits_{\Delta\in\mathcal{D}%
}S_{\Delta}(k),  \label{def total var}
\end{equation}
where $\Delta:0=t_{0}<t_{1}<\cdots<t_{n}=T$.

Write%
\begin{equation*}
BV(\left[ 0,T\right] ;\mathbb{R}^{d})=\{k:\left[ 0,T\right] \rightarrow
\mathbb{R}^{d}:\left\updownarrow k\right\updownarrow _{T}<\infty\}.
\end{equation*}
The space $BV(\left[ 0,T\right] ;\mathbb{R}^{d})$ equipped with the norm $%
\left\vert \left\vert k\right\vert \right\vert _{BV(\left[ 0,T\right] ;%
\mathbb{R}^{d})}:=\left\vert k\left( 0\right) \right\vert +\left\updownarrow
k\right\updownarrow _{T}$ is a Banach space.

Moreover, we have the duality%
\begin{equation*}
\big(C(\left[ 0,T\right] ;\mathbb{R}^{d})\big)^{\ast }=\{k\in BV(\left[ 0,T%
\right] ;\mathbb{R}^{d}):k(0)=0\}
\end{equation*}%
given by the Riemann--Stieltjes integral
\begin{equation*}
\left( y,k\right) \mapsto \int_{0}^{T}\left\langle y\left( t\right)
,dk\left( t\right) \right\rangle .
\end{equation*}%
We will say that a function $k\in BV_{loc}(\mathbb{R}_{+};\mathbb{R}^{d})$
if, for every $T>0$, $k\in BV(\left[ 0,T\right] ;\mathbb{R}^{d})$.

\section{Generalized Skorohod problem}

The aim of this section is to prove the existence and uniqueness result for
the following deterministic Cauchy type differential equation:%
\begin{equation}
\left\{
\begin{array}{l}
dx\left( t\right) +\partial^{-}\varphi\left( x\left( t\right) \right) \left(
dt\right) \ni dm\left( t\right) ,\quad t>0,\smallskip \\
x\left( 0\right) =x_{0}~,%
\end{array}
\right.  \label{formal GSP}
\end{equation}
where%
\begin{equation}
\begin{array}{rl}
\left( i\right) & x_{0}\in\overline{\mathrm{Dom}\left( \varphi\right) }%
,\smallskip \\
\left( ii\right) & m\in C\left( \mathbb{R}_{+};\mathbb{R}^{d}\right)
,~m\left( 0\right) =0,%
\end{array}
\label{assumpt input}
\end{equation}
and%
\begin{equation}
\begin{array}{r}
\varphi:\mathbb{R}^{d}\rightarrow(-\infty,+\infty]\text{ is a proper lower
semicontinuous}\smallskip \\
\text{and }\left( \rho,\gamma\right) \text{--semiconvex function.}%
\end{array}
\label{assumpt phi}
\end{equation}

\begin{definition}[Generalized Skorohod problem]
\label{sol GSP}A pair $\left( x,k\right) $ of continuous functions $x,k:%
\mathbb{R}_{+}\rightarrow $ $\mathbb{R}^{d}$, is a solution of equation (\ref%
{formal GSP}) if%
\begin{equation}
\begin{array}{rl}
\left( j\right) & x\left( t\right) \in \overline{\mathrm{Dom}\left( \varphi
\right) },\text{ \thinspace }\forall \,t\geq 0\text{ and }\varphi \left(
x\left( \cdot \right) \right) \in L_{loc}^{1}\left( \mathbb{R}_{+}\right)
,\smallskip \\
\left( jj\right) & k\in BV_{loc}\left( \mathbb{R}_{+};\mathbb{R}^{d}\right) ,%
\text{\ }k\left( 0\right) =0\text{,}\smallskip \\
\left( jjj\right) & x\left( t\right) +k\left( t\right) =x_{0}+m\left(
t\right) ,\text{\ }\forall \text{\ }t\geq 0,\smallskip \\
\left( jv\right) & \forall \,0\leq s\leq t,\;\forall y:\mathbb{R}%
_{+}\rightarrow \mathbb{R}^{d}\text{ continuous:}\smallskip \\
& \multicolumn{1}{r}{\displaystyle\int_{s}^{t}\left\langle y\left( r\right)
-x\left( r\right) ,dk\left( r\right) \right\rangle +\int_{s}^{t}\varphi
\left( x\left( r\right) \right) dr\leq \int_{s}^{t}\varphi \left( y\left(
r\right) \right) dr\smallskip} \\
& \multicolumn{1}{r}{+\displaystyle\int_{s}^{t}\left\vert y\left( r\right)
-x\left( r\right) \right\vert ^{2}\left( \rho dr+\gamma d\left\updownarrow
k\right\updownarrow _{r}\right) .}%
\end{array}
\label{def sol GSP}
\end{equation}
\end{definition}

In this case the pair $\left( x,k\right) $ is said to be the solution of the
generalized Skorohod problem $\left( \partial^{-}\varphi;x_{0},m\right) $
(denoted by $\left( x,k\right) =\mathcal{SP}\left(
\partial^{-}\varphi;x_{0},m\right) \,$).

If $\varphi=I_{E}$ then $\partial^{-}\varphi=N_{E}$ and we say that $\left(
x,k\right) $ is solution of the Skorohod problem $\left( E;x_{0},m\right) $
and we write $\left( x,k\right) =\mathcal{SP}\left( E;x_{0},m\right) .$

\begin{remark}
\label{notation_apartenance}The notation%
\begin{equation*}
dk\left( t\right) \in \partial ^{-}\varphi \left( x\left( t\right) \right)
\left( dt\right)
\end{equation*}%
means that $x,k:\mathbb{R}_{+}\rightarrow \mathbb{R}^{d}$ are continuous
functions satisfying conditions (\ref{def sol GSP}$-j,\ jj,\ jv$).
\end{remark}

The next result provides an equivalent condition for (\ref{def sol GSP}$-jv$%
) and will be used later in the proof of the continuity of the mapping $%
\left( x_{0},m\right) \mapsto \left( x,k\right) =\mathcal{SP}\left( \partial
^{-}\varphi ;x_{0},m\right) $ and for the main existence result in the
stochastic case.

\begin{lemma}
\label{echiv for (jv)}We suppose that $\varphi$ satisfies assumption (\ref%
{assumpt phi}) and let $x,k:\mathbb{R}_{+}\rightarrow\mathbb{R}^{d}$ be two
continuous functions satisfying (\ref{def sol GSP}$-j,\ jj$). Then the pair $%
\left( x,k\right) $ satisfies (\ref{def sol GSP}$-jv$) if and only if there
exists a continuous increasing function $A:\mathbb{R}_{+}\rightarrow \mathbb{%
R}_{+}$, such that
\begin{equation}
\begin{array}{ll}
\left( jv^{\prime}\right) & \forall\,0\leq s\leq t,\;\forall y:\mathbb{R}%
_{+}\rightarrow\mathbb{R}^{d}\text{ continuous:}\smallskip \\
\multicolumn{1}{r}{} & \multicolumn{1}{r}{\displaystyle\int_{s}^{t}\left%
\langle y\left( r\right) -x\left( r\right) ,dk\left( r\right) \right\rangle
+\int_{s}^{t}\varphi\left( x\left( r\right) \right)
dr\leq\int_{s}^{t}\varphi\left( y\left( r\right) \right) dr\smallskip} \\
\multicolumn{1}{r}{} & \multicolumn{1}{r}{+\displaystyle\int_{s}^{t}\left%
\vert y\left( r\right) -x\left( r\right) \right\vert ^{2}dA_{r}~.}%
\end{array}
\label{def sol GSP - jv}
\end{equation}
\end{lemma}

\begin{proof}
We only need to prove that (\ref{def sol GSP - jv}) $\Rightarrow$ (\ref{def
sol GSP}$-jv$).

Denote%
\begin{equation*}
Q_{r}:=r+\left\updownarrow k\right\updownarrow _{r}+A_{r}
\end{equation*}%
and let $\lambda ,\eta :\mathbb{R}_{+}\rightarrow \left[ 0,1\right] $ and $%
\theta :\mathbb{R}_{+}\rightarrow \mathbb{R}^{d}$ with $\left\vert \theta
\left( r\right) \right\vert \leq 1$, for any $r\in \mathbb{R}_{+}$, be some
measurable functions (given by the Radon--Nikodym's theorem) such that%
\begin{equation*}
dk\left( r\right) =\theta \left( r\right) dQ_{r}~,\quad dr=\lambda \left(
r\right) dQ_{r}\quad \text{and}\quad dA_{r}=\eta \left( r\right) dQ_{r}~.
\end{equation*}%
Clearly $d\left\updownarrow k\right\updownarrow _{r}=\left\vert \theta
\left( r\right) \right\vert dQ_{r}\,.$ From (\ref{def sol GSP - jv}) we
deduce that, for all $t\in \mathbb{R}_{+}\,,$ $\varepsilon >0$ and $z\in
\mathrm{Dom}\left( \varphi \right) $%
\begin{equation*}
\begin{array}{r}
\displaystyle\int_{t-\varepsilon }^{t+\varepsilon }\left\langle z-x\left(
r\right) ,\theta \left( r\right) \right\rangle dQ_{r}+\int_{t-\varepsilon
}^{t+\varepsilon }\varphi \left( x\left( r\right) \right) \lambda \left(
r\right) dQ_{r}\leq \varphi \left( z\right) \displaystyle\int_{t-\varepsilon
}^{t+\varepsilon }\lambda \left( r\right) dQ_{r}\smallskip \\
+\displaystyle\int_{t-\varepsilon }^{t+\varepsilon }\left\vert z-x\left(
r\right) \right\vert ^{2}\eta \left( r\right) dQ_{r}%
\end{array}%
\end{equation*}%
and therefore%
\begin{equation}
\begin{array}{l}
\displaystyle\langle z,\int_{t-\varepsilon }^{t+\varepsilon }\theta \left(
r\right) dQ_{r}\rangle -\int_{t-\varepsilon }^{t+\varepsilon }\left\langle
x\left( r\right) ,\theta \left( r\right) \right\rangle
dQ_{r}+\int_{t-\varepsilon }^{t+\varepsilon }\varphi \left( x\left( r\right)
\right) \lambda \left( r\right) dQ_{r}\leq \varphi \left( z\right)
\int_{t-\varepsilon }^{t+\varepsilon }\lambda \left( r\right)
dQ_{r}\smallskip \\
+\displaystyle\left\vert z\right\vert ^{2}\int_{t-\varepsilon
}^{t+\varepsilon }\eta \left( r\right) dQ_{r}-2\langle z,\int_{t-\varepsilon
}^{t+\varepsilon }x\left( r\right) \eta \left( r\right) dQ_{r}\rangle
+\int_{t-\varepsilon }^{t+\varepsilon }\left\vert x\left( r\right)
\right\vert ^{2}\eta \left( r\right) dQ_{r}\,.%
\end{array}
\label{inequality}
\end{equation}%
Multiplying by $\displaystyle\frac{1}{Q\left( \left[ t-\varepsilon
,t+\varepsilon \right] \right) }$ and using the Lebesgue--Besicovitch
differentiation theorem, we deduce, passing to the limit in the seven above
integrals, that there exists $\Gamma _{1}\subset \mathbb{R}_{+}$ with $%
\int_{\Gamma _{1}}dQ_{r}=0$, such that for all $z\in \mathrm{Dom}\left(
\varphi \right) $ and $r\in \mathbb{R}_{+}\smallsetminus \Gamma _{1}$%
\begin{equation*}
\left\langle z-x\left( r\right) ,\theta \left( r\right) \right\rangle
+\varphi \left( x\left( r\right) \right) \lambda \left( r\right) \leq
\varphi \left( z\right) \lambda \left( r\right) +\left\vert z-x\left(
r\right) \right\vert ^{2}\eta \left( r\right) .
\end{equation*}%
Hence, from the definition of the Fr\'{e}chet subdifferential we obtain%
\begin{equation*}
\theta \left( r\right) \in \partial ^{-}I_{\overline{\mathrm{Dom}\left(
\varphi \right) }}\left( x\left( r\right) \right) ,\quad \forall \,r\in
\Gamma _{2}\setminus \Gamma _{1}
\end{equation*}%
and%
\begin{equation*}
\frac{\theta \left( r\right) }{\lambda \left( r\right) }\in \partial
^{-}\varphi \left( x\left( r\right) \right) ,\quad \forall \,r\in (\mathbb{R}%
_{+}\smallsetminus \Gamma _{2})\setminus \Gamma _{1}\,,
\end{equation*}%
where $\Gamma _{2}=\left\{ r\geq 0:\lambda \left( r\right) =0\right\} $ with
$\int_{\Gamma _{2}}dr=\int_{\Gamma _{2}}\lambda \left( r\right) dQ_{r}=0.$

Since $I_{\overline{\mathrm{Dom}\left( \varphi \right) }}$ is $\left(
0,\gamma \right) $--semiconvex,%
\begin{equation*}
\langle y\left( r\right) -x\left( r\right) ,\theta \left( r\right) \rangle
\leq \gamma \left\vert \theta \left( r\right) \right\vert \left\vert y\left(
r\right) -x\left( r\right) \right\vert ^{2},\quad \forall ~r\in \Gamma
_{2}\setminus \Gamma _{1}\,.
\end{equation*}%
On the other hand, since $\varphi $ is a $\left( \rho ,\gamma \right) $%
--semiconvex function, we have for any continuous $y:\mathbb{R}%
_{+}\rightarrow \mathbb{R}^{d}$,%
\begin{equation*}
\begin{array}{r}
\big\langle y\left( r\right) -x\left( r\right) ,\dfrac{\theta \left(
r\right) }{\lambda \left( r\right) }\big\rangle+\varphi \left( x\left(
r\right) \right) \leq \varphi \left( y\left( r\right) \right) +\left\vert
y\left( r\right) -x\left( r\right) \right\vert ^{2}\Big(\rho +\gamma \Big|%
\dfrac{\theta \left( r\right) }{\lambda \left( r\right) }\Big|\Big),\medskip
\\
\forall ~r\in (\mathbb{R}_{+}\smallsetminus \Gamma _{2})\setminus \Gamma
_{1}\,.%
\end{array}%
\end{equation*}%
Therefore (with the convention $0\cdot \left( +\infty \right) =0$) we deduce
that, for all $r\in \mathbb{R}_{+}\smallsetminus \Gamma _{1},$%
\begin{equation*}
\begin{array}{r}
\left\langle y\left( r\right) -x\left( r\right) ,\theta \left( r\right)
\right\rangle +\varphi \left( x\left( r\right) \right) \lambda \left(
r\right) \leq \varphi \left( y\left( r\right) \right) \lambda \left(
r\right) \smallskip \\
+\left\vert y\left( r\right) -x\left( r\right) \right\vert ^{2}\left( \rho
\lambda \left( r\right) +\gamma \left\vert \theta \left( r\right)
\right\vert \right) .%
\end{array}%
\end{equation*}%
Integrating on $\left[ s,t\right] $ with respect to the measure $dQ_{r}$ we
infer that (\ref{def sol GSP}$-jv$) holds.\hfill $\smallskip $
\end{proof}

\begin{lemma}
\label{uniq lemma}If $dk\left( t\right) \in\partial^{-}\varphi\left( x\left(
t\right) \right) \left( dt\right) $ and $d\hat{k}\left( t\right)
\in\partial^{-}\varphi\left( \hat{x}\left( t\right) \right) \left( dt\right)
,$ then for all $0\leq$ $s\leq t:$%
\begin{equation}
\begin{array}{l}
\displaystyle\int_{s}^{t}\left\vert x\left( r\right) -\hat{x}\left( r\right)
\right\vert ^{2}\left( 2\rho dr+\gamma d\left\updownarrow
k\right\updownarrow _{r}+\gamma d\updownarrow\hat{k}\updownarrow_{r}\right)
\smallskip \\
\quad\quad\quad\quad\quad\quad+\displaystyle\int_{s}^{t}\big\langle x\left(
r\right) -\hat{x}\left( r\right) ,dk\left( r\right) -d\hat{k}\left( r\right) %
\big\rangle\geq0.%
\end{array}
\label{uniq ineq}
\end{equation}
\end{lemma}

\begin{proof}
The conclusion follows from (\ref{def sol GSP}$-jv$) written for $\left(
x,k\right) $ with $y=\hat{x}$ and for $(\hat{x},\hat{k})$ with $y=x$.

\hfill$\smallskip$
\end{proof}

\begin{notation}
Let $\left\Vert x\right\Vert _{\left[ s,t\right] }:={\sup\limits_{r\in \left[
s,t\right] }|x_{r}|}$ and $\left\Vert x\right\Vert _{t}:=\left\Vert
x\right\Vert _{\left[ 0,t\right] }~.$
\end{notation}

\begin{theorem}[Uniqueness]
\label{uniq}Let assumptions (\ref{assumpt input}) and (\ref{assumpt phi}) be
satisfied. If $\left( x,k\right) =\mathcal{SP}\left(
\partial^{-}\varphi;x_{0},m\right) $ and $(\hat{x},\hat{k})=\mathcal{SP}%
(\partial^{-}\varphi;\hat{x}_{0},\hat{m})$ then for all $t\geq0:$%
\begin{equation}
\begin{array}{r}
\left\Vert x-\hat{x}\right\Vert _{t}^{2}\leq2\Big(\left\vert x_{0}-\hat{x}%
_{0}\right\vert ^{2}+\left\Vert m-\hat{m}\right\Vert _{t}^{2}+2\left\Vert m-%
\hat{m}\right\Vert _{t}\Vert k-\hat{k}\Vert_{t}\Big)\smallskip \\
\cdot\exp(8\rho t+4\gamma\left\updownarrow k\right\updownarrow
_{t}+4\gamma\updownarrow\hat{k}\updownarrow_{t})~.%
\end{array}
\label{uniq ineq 2}
\end{equation}
In particular the uniqueness of the problem $\mathcal{SP}\left( \partial
^{-}\varphi;x_{0},m\right) $ follows.
\end{theorem}

\begin{proof}
We clearly have%
\begin{equation*}
\begin{array}{l}
\left\vert x\left( t\right) -m\left( t\right) -\hat{x}\left( t\right) +\hat{m%
}\left( t\right) \right\vert ^{2}=\left\vert x_{0}-\hat{x}_{0}\right\vert
^{2}\smallskip \\
+2\displaystyle\int_{0}^{t}\langle m\left( r\right) -\hat{m}\left( r\right)
,dk\left( r\right) -d\hat{k}\left( r\right) \rangle -2\displaystyle%
\int_{0}^{t}\langle x\left( r\right) -\hat{x}\left( r\right) ,dk\left(
r\right) -d\hat{k}\left( r\right) \rangle.%
\end{array}%
\end{equation*}
Using (\ref{uniq ineq}) it follows that%
\begin{equation*}
\begin{array}{l}
\frac{1}{2}\left\vert x\left( t\right) -\hat{x}\left( t\right) \right\vert
^{2}-\left\vert m\left( t\right) -\hat{m}\left( t\right) \right\vert
^{2}\leq\left\vert x\left( t\right) -m\left( t\right) -\hat{x}\left(
t\right) +\hat{m}\left( t\right) \right\vert ^{2}\smallskip \\
\leq\left\vert x_{0}-\hat{x}_{0}\right\vert ^{2}+2\left\Vert m-\hat {m}%
\right\Vert _{t}\updownarrow k-\hat{k}\updownarrow_{t}\smallskip \\
\quad\quad\quad\quad\quad\quad\quad\quad+2\displaystyle\int_{0}^{t}\left%
\vert x\left( r\right) -\hat{x}\left( r\right) \right\vert ^{2}\left( 2\rho
dr+\gamma d\left\updownarrow k\right\updownarrow _{r}+\gamma d\updownarrow
\hat{k}\updownarrow_{r}\right) ,%
\end{array}%
\end{equation*}
which implies, via Gronwall's inequality, the desired conclusion.\hfill $%
\smallskip$
\end{proof}

To derive the uniform boundedness and the continuity of the solution of the
generalized Skorohod problem we need to introduce some additional
assumptions:%
\begin{equation}
\left\vert \varphi\left( x\right) -\varphi\left( y\right) \right\vert \leq
L+L\left\vert x-y\right\vert ,\;\forall~x,y\in\mathrm{Dom}\left(
\varphi\right)  \label{assumpt phi 2}
\end{equation}
and%
\begin{equation}
\mathrm{Dom}\left( \varphi\right) \;\text{satisfies the }\left(
\gamma,\delta,\sigma\right) \text{--shifted uniform interior ball condition}
\label{assumpt phi 3}
\end{equation}
(for the definition of $\left( \gamma,\delta,\sigma\right) $--\textit{SUIBC}%
, see definition (\ref{SUIBC cond})).

We mention that assumption (\ref{assumpt phi 2}) is obviously satisfied by
the function $\varphi $ given in the Example \ref{example_semiconvex}.

Using Proposition \ref{cond for SUIBC} we see that assumption (\ref{assumpt
phi 3}) is fulfilled if we impose that%
\begin{equation}
\mathrm{Dom}\left( \varphi \right) \;\text{satisfies the }\left(
h_{0},r_{0}\right) \text{--drop condition,}  \label{assumpt phi 4}
\end{equation}%
condition which can be more easily visualized.

Note that the lower semicontinuity of $\varphi $ and the assumption (\ref%
{assumpt phi 2}) clearly yield that the $\mathrm{Dom}\left( \varphi \right) $
is a closed set, and, from the assumption (\ref{assumpt phi 3}) it can be
derived that%
\begin{equation*}
\mathrm{Int}\left( \mathrm{Dom}\left( \varphi \right) \right) \neq \emptyset
.
\end{equation*}

\begin{remark}
\label{remark 2}Technical condition (\ref{SUIBC cond}) from assumption (\ref%
{assumpt phi 3}) will provide an estimate for the total variation of $k$
(see Lemma \ref{difference of var 2}). On the other hand, assumption $\left(
5\right) $ from P--L. Lions and A.S. Sznitman \cite{li-sn/84} (or \cite[%
\textit{Condition (B)}]{sa/87}) has the same role (see \cite[Lemma 1.2]%
{li-sn/84}).

In the particular case $\varphi =I_{E}\,,$ as in \cite{li-sn/84}, it is
essentially used the representation of the bounded variation process $k$ and
in our case it is used the subdifferential inequality (\ref{def sol GSP}$-jv$%
). Hence assumption (\ref{assumpt phi 3}) is required by the transition from
the particular case of the indicator function to the case of a general
convex l.s.c. function $\varphi $.
\end{remark}

\begin{remark}
\label{remark 3}We notice that assumption (\ref{assumpt phi 4}) is similar
with \textit{Condition (B') from }\cite{sa/87} (the uniform interior cone
condition). But the running direction from the drop condition (\ref{assumpt
phi 4}) is not required to be uniform with respect to the vertex, like in
\cite[\textit{Condition (B')}]{sa/87}.
\end{remark}

In order to prove some a priori estimates, let us introduce the following
notation: for $y\in C\left( \left[ 0,T\right] ;\mathbb{R}^{d}\right) $ and $%
\varepsilon >0$ write%
\begin{equation*}
\boldsymbol{\upmu}_{y}\left( \varepsilon \right) =\varepsilon +\mathbf{m}%
_{y}\left( \varepsilon \right) ,
\end{equation*}%
where $\mathbf{m}_{y}\left( \varepsilon \right) $ is the modulus of
continuity, given by%
\begin{equation*}
\mathbf{m}_{y}\left( \varepsilon \right) :=\sup \big\{\left\vert y\left(
t\right) -y\left( s\right) \right\vert :\left\vert t-s\right\vert \leq
\varepsilon ,\;t,s\in \left[ 0,T\right] \big\}.
\end{equation*}%
The function $\boldsymbol{\upmu}_{y}:\left[ 0,T\right] \rightarrow \lbrack 0,%
\boldsymbol{\upmu}_{y}\left( T\right) ]$ is a strictly increasing continuous
function and therefore the inverse function $\boldsymbol{\upmu}_{y}^{-1}:[0,%
\boldsymbol{\upmu}_{y}\left( T\right) ]\rightarrow \left[ 0,T\right] $ is
well defined and is also a strictly increasing continuous function. Using
this inverse function let $C$ be a positive constant and%
\begin{equation}
\begin{array}{l}
\Delta _{m}:=1/\boldsymbol{\upmu}_{m}^{-1}\left( \delta ^{2}\exp [-C\left(
1+T+\left\Vert m\right\Vert _{T}\right) ]\right) ,\smallskip \\
C_{T,m}:=\exp \left[ C\left( 1+T+\left\Vert m\right\Vert _{T}+\Delta
_{m}\right) \right] .%
\end{array}
\label{main result 2_def}
\end{equation}

\begin{remark}
\label{remark 1}It is easy to prove that, for any compact subset $\mathcal{M}%
\subset C\left( \left[ 0,T\right] ;\mathbb{R}^{d}\right) $,%
\begin{equation}
\Delta_{\mathcal{M}}:=\sup_{m\in\mathcal{M}}\Delta_{m}<\infty\quad \text{and}%
\quad C_{T,\mathcal{M}}:=\sup_{m\in\mathcal{M}}C_{T,m}<\infty.
\label{technical result 1}
\end{equation}
\end{remark}

The main results of this section are the following two theorems whose proofs
will be given in the next section:

\begin{theorem}
\label{main result 2}Let assumptions (\ref{assumpt input}), (\ref{assumpt
phi}), (\ref{assumpt phi 2}) and (\ref{assumpt phi 3}) be satisfied. Then
there exists a constant $C$, depending only on the constants from the
assumptions, such that if $\left( x,k\right) =\mathcal{SP}\left( \partial
^{-}\varphi ;x_{0},m\right) $ then%
\begin{equation}
\begin{array}{ll}
\left( a\right) & \left\Vert k\right\Vert _{BV\left( \left[ 0,T\right] ;%
\mathbb{R}^{d}\right) }=\left\updownarrow k\right\updownarrow _{T}\leq
C_{T,m}~,\smallskip \\
\left( b\right) & \left\Vert x\right\Vert _{T}\leq \left\vert
x_{0}\right\vert +C_{T,m}~,\smallskip \\
\left( c\right) & \left\vert x\left( t\right) -x\left( s\right) \right\vert
+\left\updownarrow k\right\updownarrow _{t}-\left\updownarrow
k\right\updownarrow _{s}\leq C_{T,m}\cdot \sqrt{\boldsymbol{\upmu}_{m}\left(
t-s\right) },\quad \forall ~0\leq s\leq t\leq T.%
\end{array}
\label{main result 2_ineq 1}
\end{equation}%
If moreover $\hat{m}\in C\left( \left[ 0,T\right] ;\mathbb{R}^{d}\right) $, $%
\hat{x}_{0}\in \overline{\mathrm{Dom}\left( \varphi \right) }$ and $(\hat{x},%
\hat{k})=\mathcal{SP}(\partial ^{-}\varphi ;\hat{x}_{0},\hat{m}),$ then%
\begin{equation}
\Vert x-\hat{x}\Vert _{T}+\Vert k-\hat{k}\Vert _{T}\leq A\left( C_{T,m},C_{T,%
\hat{m}}\right) \cdot \big[\left\vert x_{0}-\hat{x}_{0}\right\vert +\sqrt{%
\Vert m-\hat{m}\Vert _{T}}\big],  \label{main result 2_ineq 2}
\end{equation}%
where $A$ is a continuous function.$\smallskip $
\end{theorem}

We can now derive the following continuity result of the mapping $\left(
x_{0},m\right) \mapsto \left( x,k\right) =\mathcal{SP}\left( \partial
^{-}\varphi ;x_{0},m\right) .$

\begin{corollary}
\label{main result 3}Let assumptions (\ref{assumpt input}), (\ref{assumpt
phi}), (\ref{assumpt phi 2}) and (\ref{assumpt phi 3}) be satisfied. If $%
x_{0n},x_{0}\in\overline{\mathrm{Dom}\left( \varphi\right) }$, $m_{n},m\in
C\left( \mathbb{R}_{+};\mathbb{R}^{d}\right) $, $m_{n}\left( 0\right) =0$ and%
\begin{equation*}
\begin{array}{rl}
i) & \left( x_{n},k_{n}\right) =\mathcal{SP}\left(
\partial^{-}\varphi;x_{0n},m_{n}\right) ,\smallskip \\
ii) & x_{0n}\rightarrow x_{0},\smallskip \\
iii) & m_{n}\rightarrow m\text{ in }C\left( \left[ 0,T\right] ;\mathbb{R}%
^{d}\right) ,\;\forall T\geq0,%
\end{array}%
\end{equation*}
then%
\begin{equation*}
\sup\limits_{n\in\mathbb{N}^{\ast}}\left\updownarrow k_{n}\right\updownarrow
_{T}<\infty,\;\forall T\geq0,
\end{equation*}
and there exist $x,k\in C\left( \mathbb{R}_{+};\mathbb{R}^{d}\right) $ such
that, for any $T\geq0,$%
\begin{equation*}
\begin{array}{ll}
\left( a\right) & \left\Vert x_{n}-x\right\Vert _{T}+\left\Vert
k_{n}-k\right\Vert _{T}\rightarrow0,\smallskip \\
\left( b\right) & \left( x,k\right) =\mathcal{SP}\left( \partial
^{-}\varphi;x_{0},m\right) .%
\end{array}%
\end{equation*}
\end{corollary}

\begin{proof}
Let us fix arbitrary $T>0$. The set $\mathcal{M}=\left\{ m,m_{n}:n\in
\mathbb{N}^{\ast }\right\} $ is a compact subset of $C\left( \left[ 0,T%
\right] ;\mathbb{R}^{d}\right) $. If $C_{T,m}$ is the constant defined by (%
\ref{main result 2_def}), then, using (\ref{technical result 1}), it follows
that%
\begin{equation*}
C_{T,\mathcal{M}}:=\sup_{\nu \in \mathcal{M}}C_{T,\nu }<\infty .
\end{equation*}%
Also%
\begin{equation*}
\boldsymbol{\upmu}_{\mathcal{M}}\left( \varepsilon \right) :=\sup_{\nu \in
\mathcal{M}}\boldsymbol{\upmu}_{\nu }\left( \varepsilon \right) \searrow 0,\;%
\text{as\ }\varepsilon \searrow 0.
\end{equation*}%
Let $a>0$ be such that $\left\vert x_{0n}\right\vert \leq a$. By the
estimates established in Theorem \ref{main result 2} we obtain: for all $%
n,l\in \mathbb{N}^{\ast }$ and for all $s,t\in \left[ 0,T\right] ,$ $s\leq t$%
,%
\begin{equation*}
\begin{array}{l}
\left\Vert x_{n}\right\Vert _{T}+\left\updownarrow k_{n}\right\updownarrow
_{T}\leq a+C_{T,\mathcal{M}}~,\smallskip \\
\left\vert x_{n}\left( t\right) -x_{n}\left( s\right) \right\vert
+\left\updownarrow k_{n}\right\updownarrow _{t}-\left\updownarrow
k_{n}\right\updownarrow _{s}\leq C_{T,\mathcal{M}}\cdot \sqrt{\boldsymbol{%
\upmu}_{\mathcal{M}}\left( t-s\right) }%
\end{array}%
\end{equation*}%
and%
\begin{equation*}
\left\Vert x_{n}-x_{l}\right\Vert _{T}+\left\Vert k_{n}-k_{l}\right\Vert
_{T}\leq C_{T,\mathcal{M}}\cdot \big[\left\vert x_{0n}-x_{0l}\right\vert +%
\sqrt{||m_{n}-m_{l}||_{T}}~\big].
\end{equation*}%
Therefore there exist $x,k,A\in C\left( \mathbb{R}_{+};\mathbb{R}^{d}\right)
$ such that%
\begin{equation*}
x_{n}\rightarrow x,\quad k_{n}\rightarrow k,\text{ in }C(\left[ 0,T\right] ;%
\mathbb{R}^{d})\text{, as }n\rightarrow \infty ,
\end{equation*}%
and, by Arzel\`{a}--Ascoli's Theorem, on a subsequence denoted also with $%
\left\updownarrow k_{n}\right\updownarrow $,%
\begin{equation*}
\left\updownarrow k_{n}\right\updownarrow \rightarrow A\text{, in }C(\left[
0,T\right] ;\mathbb{R}^{d})\text{, as }n\rightarrow \infty ,
\end{equation*}%
where $A$ is an increasing function starting from zero.

Clearly, the pair $\left( x,k\right) $ satisfies (\ref{def sol GSP}$%
-j,jj,jjj $) and (\ref{def sol GSP - jv}), which means, using Lemma \ref%
{echiv for (jv)}, that $\left( x,k\right) =\mathcal{SP}\left(
\partial^{-}\varphi;x_{0},m\right) $.\hfill$\smallskip$
\end{proof}

\begin{theorem}
\label{main result 1}Let assumptions (\ref{assumpt input}), (\ref{assumpt
phi}), (\ref{assumpt phi 2}) and (\ref{assumpt phi 3}) be satisfied.

Then the generalized Skorohod problem%
\begin{equation*}
\left\{
\begin{array}{l}
x\left( t\right) +k\left( t\right) =x_{0}+m\left( t\right) ,\quad t\geq
0,\smallskip \\
dk\left( t\right) \in \partial ^{-}\varphi \left( x\left( t\right) \right)
\left( dt\right)%
\end{array}%
\right.
\end{equation*}%
has a unique solution $\left( x,k\right) $, in the sense of Definition \ref%
{sol GSP} (and we write $\left( x,k\right) =\mathcal{SP}\left( \partial
^{-}\varphi ;x_{0},m\right) $).
\end{theorem}

Before giving the proof of the main results, Theorems \ref{main result 2}
and \ref{main result 1}, let us examine the particular case of the indicator
function of the closed set $E$ (which yields the classical Skorohod problem).

If $E$ satisfies the $r_{0}-UEBC$, then, by Lemmas \ref{UEBC - semiconvex}
and \ref{normal cone} and Definition \ref{semiconvex function}, the set $E$
is $\tfrac{1}{2r_{0}}$--semiconvex and the indicator function $I_{E}\left(
x\right) $ is a $(0,\tfrac{1}{2r_{0}})$--semiconvex function. Hence
assumptions (\ref{assumpt phi}) and (\ref{assumpt phi 2}) are satisfied.

We write the definition of the solution of the Skorohod problem in the case
of the indicator function.

\begin{definition}[Skorohod problem]
\label{sol SP}Let $E\subset\mathbb{R}^{d}$ be a set satisfying the $%
r_{0}-UEBC$. A pair $\left( x,k\right) $ is a solution of the Skorohod
problem if $x,k:\mathbb{R}_{+}\rightarrow$ $\mathbb{R}^{d}$ are continuous
functions and%
\begin{equation}
\begin{array}{rl}
\left( j\right) & x\left( t\right) \in E,\smallskip \\
\left( jj\right) & k\in BV_{loc}\left( \mathbb{R}_{+};\mathbb{R}^{d}\right) ,%
\text{\ }k\left( 0\right) =0\text{,}\smallskip \\
\left( jjj\right) & x\left( t\right) +k\left( t\right) =x_{0}+m\left(
t\right) ,\smallskip \\
\left( jv\right) & \forall~0\leq s\leq t\leq T,\;\forall y\in C\left(
\mathbb{R}_{+};E\right) \smallskip \\
& \displaystyle\int_{s}^{t}\left\langle y\left( r\right) -x\left( r\right)
,dk\left( r\right) \right\rangle \leq\dfrac{1}{2r_{0}}\int_{s}^{t}\left\vert
y\left( r\right) -x\left( r\right) \right\vert ^{2}d\left\updownarrow
k\right\updownarrow _{r}~.%
\end{array}
\label{def sol SP}
\end{equation}
\end{definition}

The following theorem is a consequence of the main existence Theorem \ref%
{main result 1}.

\begin{theorem}
\label{main result 4}Let $x_{0}\in E$ and $m:\mathbb{R}_{+}\rightarrow
\mathbb{R}^{d}$ be a continuous function such that $m\left( 0\right) =0$. If
$E$ satisfies the $r_{0}-UEBC$, for some $r_{0}>0$, and the shifted uniform
interior ball condition, then there exists a unique solution of the Skorohod
problem, in the sense of Definition \ref{sol SP}.

Moreover, each one of the following conditions are equivalent with (\ref{def
sol SP}$-jv$):%
\begin{equation}
\left( jv^{\prime}\right) \left\{
\begin{array}{l}
\left\updownarrow k\right\updownarrow _{t}=\displaystyle\int_{0}^{t}\mathbf{1%
}_{x\left( s\right) \in\mathrm{Bd}\left( E\right) }d\left\updownarrow
k\right\updownarrow _{s}~,\smallskip \\
k\left( t\right) =\displaystyle\int_{0}^{t}n_{x\left( s\right)
}d\left\updownarrow k\right\updownarrow _{s},\;\smallskip \\
\quad\quad\text{where }n_{x\left( s\right) }\in N_{E}\left( x\left( s\right)
\right) \;\text{and }|n_{x\left( s\right) }|=1,\;d\left\updownarrow
k\right\updownarrow _{s}\text{--a.e.,}%
\end{array}
\right.  \label{def sol SP - jv}
\end{equation}
and%
\begin{equation}
\begin{array}{ll}
\left( jv^{\prime\prime}\right) & \exists~\beta>0\text{ such that }\forall~y:%
\mathbb{R}_{+}\rightarrow E\text{ continuous,}\smallskip \\
& \quad\displaystyle\int_{s}^{t}\left\langle y\left( r\right) -x\left(
r\right) ,dk\left( r\right) \right\rangle \leq\beta\int_{s}^{t}\left\vert
y\left( r\right) -x\left( r\right) \right\vert ^{2}d\left\updownarrow
k\right\updownarrow _{r}.%
\end{array}
\label{def sol SP - jv'}
\end{equation}
\end{theorem}

\begin{proof}
The uniqueness is ensures by Theorem \ref{main result 1}. For a direct proof
it is sufficient to use inequality \ref{uniq ineq 2} and the next one: if $%
\left( x,k\right) =\mathcal{SP}\left( E;x_{0},m\right) $ and $(\hat{x},\hat{k%
})=\mathcal{SP}\left( E;\hat{x}_{0},\hat{m}\right) $, then from (\ref{def
sol SP}) we get (see also Lemma \ref{uniq lemma})%
\begin{equation*}
\langle x\left( t\right) -\hat{x}\left( t\right) ,dk\left( t\right) -d\hat{k}%
\left( t\right) \rangle +\dfrac{1}{2r_{0}}\left\vert x\left( t\right) -\hat{x%
}\left( t\right) \right\vert ^{2}(d\left\updownarrow k\right\updownarrow
_{t}+d\updownarrow \hspace{-0.09cm}\hat{k}\hspace{-0.05cm}\updownarrow
_{t})\geq 0.
\end{equation*}

The existence is due to Theorem \ref{main result 1} (but, for a direct proof
of the existence, with condition (\ref{def sol SP}$-jv$) replaced by (\ref%
{def sol SP - jv}), we refer the reader to \cite{li-sn/84}).

Proof of (\ref{def sol SP - jv}) $\Longrightarrow $ (\ref{def sol SP}$-jv$):
using Lemma \ref{UEBC - semiconvex} we see%
\begin{equation*}
\begin{array}{l}
\displaystyle\int_{s}^{t}\left\langle y\left( r\right) -x\left( r\right)
,dk\left( r\right) \right\rangle =\int_{s}^{t}\left\langle y\left( r\right)
-x\left( r\right) ,n_{x\left( r\right) }d\left\updownarrow
k\right\updownarrow _{r}\right\rangle \smallskip \\
=\displaystyle\int_{s}^{t}\left\langle y\left( r\right) -x\left( r\right)
,n_{x\left( r\right) }\mathbf{1}_{x\left( r\right) \in \mathrm{Bd}\left(
E\right) }d\left\updownarrow k\right\updownarrow _{r}\right\rangle \smallskip
\\
\quad \leq \dfrac{1}{2r_{0}}\displaystyle\int_{s}^{t}\left\vert n_{x\left(
r\right) }\right\vert \left\vert y\left( r\right) -x\left( r\right)
\right\vert ^{2}\mathbf{1}_{x\left( s\right) \in \mathrm{Bd}\left( E\right)
}d\left\updownarrow k\right\updownarrow _{r}\smallskip \\
\quad \leq \dfrac{1}{2r_{0}}\displaystyle\int_{s}^{t}\left\vert y\left(
r\right) -x\left( r\right) \right\vert ^{2}d\left\updownarrow
k\right\updownarrow _{r}~.%
\end{array}%
\end{equation*}

Clearly (\ref{def sol SP}$-jv$) $\Longrightarrow$ (\ref{def sol SP - jv'}).

Proof of (\ref{def sol SP - jv'}) $\Longrightarrow $ (\ref{def sol SP - jv}%
): let $\left[ s,t\right] $ be an interval such that $x\left( r\right) \in
\mathrm{Int}\left( E\right) $ for all $r\in \lbrack s,t]$. Then there exists
$\delta =\delta _{s,t}>0$ such that%
\begin{equation*}
\inf_{r\in \left[ s,t\right] }\mathrm{d}_{\mathrm{Bd}\left( E\right) }\big(%
x\left( r\right) \big)\geq \delta .
\end{equation*}%
Let $\lambda \in \left[ 0,\delta \right] $ and $\alpha \in C\left( \left[ 0,T%
\right] ;\mathbb{R}^{d}\right) $ such that $\left\Vert \alpha \right\Vert
_{T}\leq 1$. Setting $y\left( r\right) =x\left( r\right) +\lambda \alpha
\left( r\right) $ in (\ref{def sol SP - jv'}) we obtain%
\begin{equation*}
\int_{s}^{t}\left\langle \alpha \left( r\right) ,dk\left( r\right)
\right\rangle \leq \beta \lambda \int_{s}^{t}d\left\updownarrow
k\right\updownarrow _{r}~.
\end{equation*}%
Hence, passing to the limit, for $\lambda \rightarrow 0$, and taking $%
\sup_{\left\Vert \alpha \right\Vert _{T}\leq 1}$, we deduce the implication:%
\begin{equation}
x\left( r\right) \in \mathrm{Int}\left( E\right) ,\;\forall r\in \lbrack
s,t]\;\Longrightarrow \;\left\updownarrow k\right\updownarrow
_{t}-\left\updownarrow k\right\updownarrow _{s}=0.  \label{technical ineq 1}
\end{equation}%
Let $\ell \left( r\right) $ be a measurable function such that $\left\vert
\ell \left( r\right) \right\vert =1$, $d\left\updownarrow
k\right\updownarrow _{r}$--a.e. and%
\begin{equation*}
k\left( t\right) =\int_{0}^{t}\ell \left( r\right) d\left\updownarrow
k\right\updownarrow _{r}~.
\end{equation*}%
Since (\ref{def sol SP - jv'}) holds for all $0\leq s\leq t$, we deduce,
using the Lebesgue--Besicovitch theorem, that%
\begin{equation*}
\beta \left\vert y\left( r\right) -x\left( r\right) \right\vert
^{2}-\left\langle \ell \left( r\right) ,y\left( r\right) -x\left( r\right)
\right\rangle \geq 0,\quad d\left\updownarrow k\right\updownarrow _{r}\text{%
--a.e.,}
\end{equation*}%
for any $y\in C\left( \left[ 0,T\right] ;E\right) $.

Therefore, from Lemma \ref{normal cone}, we infer that%
\begin{equation}
\ell \left( r\right) \in N_{E}\left( x\left( r\right) \right)
,\;d\left\updownarrow k\right\updownarrow _{r}\text{--a.e.}
\label{technical ineq 2}
\end{equation}%
We have thus proved inequality (\ref{def sol SP - jv}), since we have (\ref%
{technical ineq 1}) and (\ref{technical ineq 2}).\hfill
\end{proof}

\section{Generalized Skorohod problem: proofs}

In order to prove Theorem \ref{main result 1} let us first prove some
auxiliary results.

Let $\left( x,k\right) =\mathcal{SP}\left( \partial ^{-}\varphi
;x_{0},m\right) $ and $y\in C\left( \mathbb{R}_{+};E\right) $, where $E=%
\mathrm{Dom}\left( \varphi \right) $. From (\ref{assumpt phi 2}) and (\ref%
{def sol GSP}--$jv$) we have, for all $0\leq s\leq t,$%
\begin{equation}
\begin{array}{l}
\displaystyle\int_{s}^{t}\left\langle y\left( r\right) -x\left( r\right)
,dk\left( r\right) \right\rangle \leq L\left( t-s\right)
+L\int_{s}^{t}\left\vert y\left( r\right) -x\left( r\right) \right\vert
dr\medskip \\
\quad \quad \quad \quad \quad \quad \quad \quad \quad \quad \quad \quad
\quad \quad \quad \quad \quad +\displaystyle\int_{s}^{t}\left\vert y\left(
r\right) -x\left( r\right) \right\vert ^{2}\left( \rho dr+\gamma
d\left\updownarrow k\right\updownarrow _{r}\right) .%
\end{array}
\label{difference of var_ineq 1}
\end{equation}%
Suppose that $x\left( r\right) \in \mathrm{Int}\left( \mathrm{Dom}\left(
\varphi \right) \right) $ for all $r\in \left[ s,t\right] ,$ and let%
\begin{equation*}
0<b\leq \inf_{r\in \left[ s,t\right] }\mathrm{d}_{\mathrm{Bd}\left( E\right)
}\left( x\left( r\right) \right) .
\end{equation*}%
Write $y\left( r\right) =x\left( r\right) +\lambda b\alpha \left( r\right) $
with $\alpha \in C\left( \mathbb{R}_{+};\mathbb{R}^{d}\right) ,$ $\left\Vert
\alpha \right\Vert _{\left[ s,t\right] }\leq 1$ and $0<\lambda <1.$ Hence
the above inequality becomes, for $\lambda =\left[ \left( 1+\gamma \right)
\left( 1+b\right) ^{2}\right] ^{-1}$%
\begin{align*}
\lambda b\int_{s}^{t}\left\langle \alpha \left( r\right) ,dk\left( r\right)
\right\rangle & \leq \left( L+Lb\right) \left( t-s\right) +\lambda ^{2}b^{2}
\left[ \rho \left( t-s\right) +\gamma \left( \left\updownarrow
k\right\updownarrow _{t}-\left\updownarrow k\right\updownarrow _{s}\right) %
\right] \\
& \leq \left( L+Lb+\lambda ^{2}b^{2}\rho \right) \left( t-s\right) +\frac{%
\lambda b}{1+b}\left( \left\updownarrow k\right\updownarrow
_{t}-\left\updownarrow k\right\updownarrow _{s}\right) .
\end{align*}%
Taking the supremum over all $\alpha $ such that $\left\Vert \alpha
\right\Vert _{\left[ s,t\right] }\leq 1$ we see that%
\begin{equation*}
\frac{\lambda b^{2}}{1+b}\left( \left\updownarrow k\right\updownarrow
_{t}-\left\updownarrow k\right\updownarrow _{s}\right) \leq \left(
L+Lb+\lambda ^{2}b^{2}\rho \right) \left( t-s\right) .
\end{equation*}%
Consequently, the following result is proved:

\begin{lemma}
\label{difference of var 1}Let $\varphi $ such that assumption (\ref{assumpt
phi 2}) is satisfied and $\left( x,k\right) =\mathcal{SP}\left( \partial
^{-}\varphi ;x_{0},m\right) $. If $x\left( r\right) \in \mathrm{Int}\left(
\mathrm{Dom}\left( \varphi \right) \right) $, for all $r\in \left[ s,t\right]
$, then there exists a positive constant $C=C\left( L,\rho ,\gamma ,b\right)
$ such that%
\begin{equation*}
\left\updownarrow k\right\updownarrow _{t}-\left\updownarrow
k\right\updownarrow _{s}\leq C\left( t-s\right)
\end{equation*}%
where%
\begin{equation*}
0<b\leq \inf_{r\in \left[ s,t\right] }\mathrm{d}_{\mathrm{Bd}\left( E\right)
}\left( x\left( r\right) \right) .
\end{equation*}
\end{lemma}

More generally we have

\begin{lemma}
\label{difference of var 2}Let $\varphi$ be a $\left( \rho,\gamma\right) $%
--semiconvex function and $\left( x,k\right) =\mathcal{SP}\left(
\partial^{-}\varphi;x_{0},m\right) $. Assume that $\varphi$ satisfies
assumption (\ref{assumpt phi 2}) and set $\mathrm{Dom}\left( \varphi\right) $
satisfies assumption (\ref{assumpt phi 3}). If $0\leq s\leq t$ and%
\begin{equation*}
\sup_{r\in\left[ s,t\right] }\left\vert x\left( r\right) -x\left( s\right)
\right\vert \leq\delta,
\end{equation*}
then there exists $\sigma>0$ such that%
\begin{equation}
\left\updownarrow k\right\updownarrow _{t}-\left\updownarrow
k\right\updownarrow _{s}\leq\frac{1}{\sigma}\left\vert k\left( t\right)
-k\left( s\right) \right\vert +\frac{3L+4\rho}{\sigma}\left( t-s\right) .
\label{difference of var_ineq 2}
\end{equation}
\end{lemma}

\begin{proof}
Let us fix arbitrarily $\alpha\in C\left( \mathbb{R}_{+};\mathbb{R}%
^{d}\right) $ such that $\left\Vert \alpha\right\Vert _{\left[ s,t\right]
}\leq1$. From assumptions (\ref{assumpt phi 2}--\ref{assumpt phi 3}), if%
\begin{equation*}
y\left( r\right) :=x\left( r\right) +v_{x\left( s\right) }+\lambda_{x\left(
s\right) }\alpha\left( r\right) ,\;\;r\in\left[ s,t\right] ,
\end{equation*}
then $y\left( r\right) \in E$.

Moreover%
\begin{equation*}
\left\vert y\left( r\right) -x\left( r\right) \right\vert \leq\left\vert
v_{x\left( s\right) }\right\vert +\lambda_{x\left( s\right) }\leq2
\end{equation*}
and%
\begin{equation*}
\left\vert \varphi\left( y\left( r\right) \right) -\varphi\left( x\left(
r\right) \right) \right\vert \leq3L.
\end{equation*}
From (\ref{difference of var_ineq 1}) we deduce that%
\begin{align*}
\lambda_{x\left( s\right) }\int_{s}^{t}\left\langle \alpha\left( r\right)
,dk\left( r\right) \right\rangle & \leq-\int_{s}^{t}\left\langle v_{x\left(
s\right) },dk\left( r\right) \right\rangle +\left( 3L+4\rho\right) \left(
t-s\right) \\
& +\gamma\int_{s}^{t}\left( \left\vert v_{x\left( s\right) }\right\vert
+\lambda_{x\left( s\right) }\right) ^{2}d\left\updownarrow
k\right\updownarrow _{r}~.
\end{align*}
Taking the supremum over all $\alpha$ such that $\left\Vert
\alpha\right\Vert _{\left[ s,t\right] }\leq1$ we see, using also (\ref{SUIBC
cond}), that%
\begin{equation*}
\sigma\left( \left\updownarrow k\right\updownarrow _{t}-\left\updownarrow
k\right\updownarrow _{s}\right) \leq\left\vert k\left( t\right) -k\left(
s\right) \right\vert +\left( 3L+4\rho\right) \left( t-s\right)
\end{equation*}
and the Lemma follows.\hfill$\smallskip$
\end{proof}

\begin{proof}[Proof of Theorem \protect\ref{main result 2}]
We denote by $C,C^{\prime},C^{\prime\prime}$ generic constants independent
of $x_{0}$, $\hat{x}_{0}$, $m$, $\hat{m}$ and $T$, but possibly depending on
constants $L,\delta ,\sigma,\rho,\gamma$ provided by the assumptions.$%
\smallskip$

\noindent\textrm{Step 1.}\textit{\ Some estimates of the modulus of
continuity of the function }$x$\textit{.}

Let $0\leq s\leq t\leq T.$

Since%
\begin{equation*}
\left\vert x\left( t\right) -x\left( s\right) -m\left( t\right) +m\left(
s\right) \right\vert =\left\vert k\left( t\right) -k\left( s\right)
\right\vert \leq\left\updownarrow k\right\updownarrow _{t}-\left\updownarrow
k\right\updownarrow _{s},
\end{equation*}
it follows that%
\begin{equation*}
\left\vert x\left( t\right) -x\left( s\right) \right\vert \leq\left\vert
m\left( t\right) -m\left( s\right) \right\vert +\left\updownarrow
k\right\updownarrow _{t}-\left\updownarrow k\right\updownarrow _{s}.
\end{equation*}
We clearly have%
\begin{align*}
\left\vert x\left( t\right) -x\left( s\right) -m\left( t\right) +m\left(
s\right) \right\vert ^{2} & =2\int_{s}^{t}\left\langle m\left( r\right)
-m\left( s\right) ,dk\left( r\right) \right\rangle \\
& +2\int_{s}^{t}\left\langle x\left( s\right) -x\left( r\right) ,dk\left(
r\right) \right\rangle .
\end{align*}
From (\ref{def sol GSP}$-jv$) written for $y\left( r\right) \equiv x\left(
s\right) $ and (\ref{assumpt phi 2}) we have%
\begin{align*}
\int_{s}^{t}\left\langle x\left( s\right) -x\left( r\right) ,dk\left(
r\right) \right\rangle & \leq L\left( t-s\right) +L\int_{s}^{t}\left\vert
x\left( s\right) -x\left( r\right) \right\vert dr \\
& +\int_{s}^{t}\left\vert x\left( s\right) -x\left( r\right) \right\vert
^{2}\left( \rho dr+\gamma d\left\updownarrow k\right\updownarrow _{r}\right)
.
\end{align*}
Thus, using also the inequality $\frac{1}{2}\left\vert \alpha\right\vert
^{2}\leq\left\vert \alpha-\beta\right\vert ^{2}+\left\vert \beta\right\vert
^{2}$, we get%
\begin{equation*}
\begin{array}{r}
\dfrac{1}{2}\left\vert x\left( t\right) -x\left( s\right) \right\vert
^{2}\leq\left\vert m\left( t\right) -m\left( s\right) \right\vert ^{2}+2%
\mathbf{m}_{m}\left( t-s\right) \left( \left\updownarrow k\right\updownarrow
_{t}-\left\updownarrow k\right\updownarrow _{s}\right) +C\left( t-s\right)
\smallskip \\
+C\displaystyle\int_{s}^{t}\left\vert x\left( r\right) -x\left( s\right)
\right\vert ^{2}\left( dr+d\left\updownarrow k\right\updownarrow _{r}\right)
,%
\end{array}%
\end{equation*}
and, by Gronwall's inequality,%
\begin{equation}
\begin{array}{r}
\left\vert x\left( t\right) -x\left( s\right) \right\vert ^{2}\leq\left[
\mathbf{m}_{m}^{2}\left( t-s\right) +\mathbf{m}_{m}\left( t-s\right) \left(
\left\updownarrow k\right\updownarrow _{t}-\left\updownarrow
k\right\updownarrow _{s}\right) +\left( t-s\right) \right] \smallskip \\
\cdot\exp\left[ C\left( 1+t-s+\left\updownarrow k\right\updownarrow
_{t}-\left\updownarrow k\right\updownarrow _{s}\right) \right] ,%
\end{array}
\label{main result 2_ineq 3}
\end{equation}
for all $\,0\leq s\leq t\leq T.\smallskip$

\noindent\textrm{Step 2.}\textit{\ Estimates of the differences }$\left\vert
x\left( t\right) -x\left( s\right) \right\vert $\textit{\ and }$%
\left\updownarrow k\right\updownarrow _{t}-\left\updownarrow
k\right\updownarrow _{s}$\textit{\ under the assumption }$\left\vert x\left(
t\right) -x\left( s\right) \right\vert \leq\delta$.

Let $0\leq s\leq r\leq t\leq T$ such that $\left\vert x\left( t\right)
-x\left( s\right) \right\vert \leq\delta$. From (\ref{difference of var_ineq
2}) we have%
\begin{equation*}
\begin{array}{l}
\left\updownarrow k\right\updownarrow _{t}-\left\updownarrow
k\right\updownarrow _{s}\leq C\left\vert k\left( t\right) -k\left( s\right)
\right\vert +C\left( t-s\right) \smallskip \\
=C\left\vert x\left( t\right) -x\left( s\right) -m\left( t\right) +m\left(
s\right) \right\vert +C\left( t-s\right) \smallskip \\
\leq C\left\vert x\left( t\right) -x\left( s\right) \right\vert +C\mathbf{m}%
_{m}\left( t-s\right) +C~\left( t-s\right) \leq C\delta +C\boldsymbol{\upmu}%
_{m}\left( t-s\right) .%
\end{array}%
\end{equation*}
Using the estimate (\ref{main result 2_ineq 3}), it clearly follows that%
\begin{equation*}
\left\vert x\left( t\right) -x\left( s\right) \right\vert ^{2}\leq%
\boldsymbol{\upmu}_{m}\left( t-s\right) \exp\left[ C\left( 1+T+\left\Vert
m\right\Vert _{T}\right) \right] ,\;\;\text{for all }0\leq s\leq t\leq T,
\end{equation*}
since $\left( t-s\right) +\mathbf{m}_{m}\left( t-s\right) =\boldsymbol{\upmu}%
_{m}\left( t-s\right) \leq T+2\left\Vert m\right\Vert _{T}~.$

Hence if $0\leq s\leq t\leq T$ and $\left\vert x\left( t\right) -x\left(
s\right) \right\vert \leq\delta$ then%
\begin{equation}
\left\vert x\left( t\right) -x\left( s\right) \right\vert +\left\updownarrow
k\right\updownarrow _{t}-\left\updownarrow k\right\updownarrow _{s}\leq\sqrt{%
\boldsymbol{\upmu}_{m}\left( t-s\right) }\cdot\exp\left[ C\left(
1+T+\left\Vert m\right\Vert _{T}\right) \right] .
\label{main result 2_ineq 4}
\end{equation}

\noindent\textrm{Step 3. }\textit{Adapted time partition and local estimates.%
}

Let the sequence given by (the definition is suggested by \cite{li-sn/84})%
\begin{align*}
t_{0}& =T_{0}=0, \\
T_{1}& =\inf \left\{ t\in \left[ t_{0},T\right] :\mathrm{d}_{\mathrm{Bd}%
\left( E\right) }\left( x\left( t\right) \right) \leq \delta /4\right\} , \\
t_{1}& =\inf \left\{ t\in \left[ T_{1},T\right] :\left\vert x\left( t\right)
-x\left( T_{1}\right) \right\vert >\delta /2\right\} , \\
T_{2}& =\inf \left\{ t\in \left[ t_{1},T\right] :\mathrm{d}_{\mathrm{Bd}%
\left( E\right) }\left( x\left( t\right) \right) \leq \delta /4\right\} , \\
& \cdots \cdots \cdots \\
t_{i}& =\inf \left\{ t\in \left[ T_{i},T\right] :\left\vert x\left( t\right)
-x\left( T_{i}\right) \right\vert >\delta /2\right\} \\
T_{i+1}& =\inf \left\{ t\in \left[ t_{i},T\right] :\mathrm{d}_{\mathrm{Bd}%
\left( E\right) }\left( x\left( t\right) \right) \leq \delta /4\right\} \\
& \cdots \cdots \cdots
\end{align*}%
Clearly%
\begin{equation*}
0=T_{0}=t_{0}\leq T_{1}<t_{1}\leq T_{2}<\cdots <t_{i}\leq
T_{i+1}<t_{i+1}\leq \cdots \leq T.
\end{equation*}%
Let $t_{i}\leq r\leq T_{i+1}$. Then $x\left( r\right) \in \mathrm{Int}\left(
E\right) $ and $\mathrm{d}_{\mathrm{Bd}\left( E\right) }\left( x\left(
r\right) \right) \geq \delta /4.$ By Lemma \ref{difference of var 1} we get%
\begin{equation*}
\left\vert k\left( t\right) -k\left( s\right) \right\vert \leq
\left\updownarrow k\right\updownarrow _{t}-\left\updownarrow
k\right\updownarrow _{s}\leq C\left( t-s\right) \;\text{for }t_{i}\leq s\leq
t\leq T_{i+1}~.
\end{equation*}%
Also for $t_{i}\leq s\leq t\leq T_{i+1}:$%
\begin{align*}
\left\vert x\left( t\right) -x\left( s\right) \right\vert & \leq \left\vert
k\left( t\right) -k\left( s\right) \right\vert +\left\vert m\left( t\right)
-m\left( s\right) \right\vert \leq C\left( t-s\right) +\left\vert m\left(
t\right) -m\left( s\right) \right\vert \\
& \leq C\boldsymbol{\upmu}_{m}\left( t-s\right)
\end{align*}%
and then%
\begin{equation*}
\left\vert x\left( t\right) -x\left( s\right) \right\vert +\left\updownarrow
k\right\updownarrow _{t}-\left\updownarrow k\right\updownarrow _{s}\leq C%
\boldsymbol{\upmu}_{m}\left( t-s\right) .
\end{equation*}%
On each of the intervals $\left[ T_{i},t_{i}\right] ,$ we have
\begin{equation*}
\left\vert x\left( t\right) -x\left( s\right) \right\vert \leq \delta ,\text{%
\quad for all }T_{i}\leq s\leq t\leq t_{i}.
\end{equation*}%
and consequently, for all $T_{i}\leq s\leq t\leq t_{i}$, inequality (\ref%
{main result 2_ineq 4}) holds.

If $T_{i}\leq s\leq t_{i}\leq t\leq T_{i+1}$ then%
\begin{equation*}
\begin{array}{l}
\left\vert x\left( t\right) -x\left( s\right) \right\vert +\left\updownarrow
k\right\updownarrow _{t}-\left\updownarrow k\right\updownarrow _{s}\smallskip
\\
\leq\left\vert x\left( t\right) -x\left( t_{i}\right) \right\vert
+\left\updownarrow k\right\updownarrow _{t}-\left\updownarrow
k\right\updownarrow _{t_{i}}+\left\vert x\left( t_{i}\right) -x\left(
s\right) \right\vert +\left\updownarrow k\right\updownarrow
_{t_{i}}-\left\updownarrow k\right\updownarrow _{s}\smallskip \\
\leq C\boldsymbol{\upmu}_{m}\left( t-t_{i}\right) +\sqrt{\boldsymbol{\upmu}%
_{m}\left( t_{i}-s\right) }\cdot\exp\left[ C\left( 1+T+\left\Vert
m\right\Vert _{T}\right) \right] \smallskip \\
\leq\sqrt{\boldsymbol{\upmu}_{m}\left( t-s\right) }\times\exp\left[
C^{\prime}\left( 1+T+\left\Vert m\right\Vert _{T}\right) \right] .%
\end{array}%
\end{equation*}
Consequently for all $i\in\mathbb{N}$ and $T_{i}\leq s\leq t\leq T_{i+1}$,
inequality (\ref{main result 2_ineq 4}) holds.$\smallskip$

\noindent\textrm{Step 4.}\textit{\ Getting inequalities }(\ref{main result
2_ineq 1}).

Since $\boldsymbol{\upmu}_{m}^{-1}:\left[ 0,\boldsymbol{\upmu}_{m}\left(
T\right) \right] \rightarrow\left[ 0,T\right] $ is well defined and is a
strictly increasing continuous function, from%
\begin{equation*}
\begin{array}{l}
\frac{\delta}{2}\leq\left\vert x\left( t_{i}\right) -x\left( T_{i}\right)
\right\vert \leq\sqrt{\boldsymbol{\upmu}_{m}\left( t_{i}-T_{i}\right) }%
\times\exp\left[ C\left( 1+T+\left\Vert m\right\Vert _{T}\right) \right]
\smallskip \\
\leq\sqrt{\boldsymbol{\upmu}_{m}\left( T_{i+1}-T_{i}\right) }\cdot \exp\left[
C\left( 1+T+\left\Vert m\right\Vert _{T}\right) \right] ,%
\end{array}%
\end{equation*}
we deduce that%
\begin{align*}
T_{i+1}-T_{i} & \geq\boldsymbol{\upmu}_{m}^{-1}\left( \frac{\delta^{2}}{4}%
\exp\left[ -2C\left( 1+T+\left\Vert m\right\Vert _{T}\right) \right] \right)
\\
& \geq\boldsymbol{\upmu}_{m}^{-1}\left( \delta^{2}\exp\left[ -2C^{\prime
}\left( 1+T+\left\Vert m\right\Vert _{T}\right) \right] \right) >0.
\end{align*}
Hence the bounded increasing sequence $\left( T_{i}\right) _{i\geq0}$ has a
finite numbers of terms, therefore there exists $j\in\mathbb{N}^{\ast}$ such
that $T=T_{j}$. Then%
\begin{equation*}
T=T_{j}=\sum_{i=1}^{j}\left( T_{i}-T_{i-1}\right) \geq j\Delta_{m}^{-1},
\end{equation*}
where $\Delta_{m}:=1/\boldsymbol{\upmu}_{m}^{-1}\left( \delta^{2}\exp\left[
-C^{\prime}\left( 1+T+\left\Vert m\right\Vert _{T}\right) \right] \right) $
(see definition (\ref{main result 2_def})).

Let $0\leq s\leq t\leq T.$ We have%
\begin{align*}
\left\updownarrow k\right\updownarrow _{t}-\left\updownarrow
k\right\updownarrow _{s} & =\sum_{i=1}^{j}\left( \left\updownarrow
k\right\updownarrow _{\left( t\wedge T_{i}\right) \vee s}-\left\updownarrow
k\right\updownarrow _{\left( t\wedge T_{i-1}\right) \vee s}\right) \\
& \leq\sum_{i=1}^{j}\sqrt{\boldsymbol{\upmu}_{m}\big(\left( t\wedge
T_{i}\right) \vee s-\left( t\wedge T_{i-1}\right) \vee s\big)}\cdot \exp%
\left[ C\left( 1+T+\left\Vert m\right\Vert _{T}\right) \right] \\
& \leq j\sqrt{\boldsymbol{\upmu}_{m}\left( t-s\right) }\cdot\exp\left[
C\left( 1+T+\left\Vert m\right\Vert _{T}\right) \right] \\
& \leq T\Delta_{m}\sqrt{\boldsymbol{\upmu}_{m}\left( t-s\right) }\cdot \exp%
\left[ C\left( 1+T+\left\Vert m\right\Vert _{T}\right) \right]
\end{align*}
and consequently%
\begin{equation*}
\left\updownarrow k\right\updownarrow _{T}\leq T\Delta_{m}\sqrt{\boldsymbol{%
\upmu}_{m}\left( T\right) }\cdot\exp\left[ C\left( 1+T+\left\Vert
m\right\Vert _{T}\right) \right] \leq\exp\left[ C^{\prime}\left(
1+T+\left\Vert m\right\Vert _{T}+\Delta_{m}\right) \right]
\end{equation*}
and
\begin{equation*}
\left\vert x\left( t\right) \right\vert =\left\vert x_{0}+m\left( t\right)
-k\left( t\right) \right\vert \leq\left\vert x_{0}\right\vert +\left\Vert
m\right\Vert _{t}+\left\updownarrow k\right\updownarrow _{t}\leq\left\vert
x_{0}\right\vert +\left\Vert m\right\Vert _{T}+\left\updownarrow
k\right\updownarrow _{T}.
\end{equation*}
Hence there exists a positive constant $C=C\left( L,\delta,\sigma,\rho
,\gamma\right) $ such that, under notations (\ref{main result 2_def}),%
\begin{equation*}
\left\updownarrow k\right\updownarrow _{T}\leq C_{T,m},\quad\text{and}%
\quad\left\Vert x\right\Vert _{T}\leq\left\vert x_{0}\right\vert +C_{T,m},
\end{equation*}
which is part of conclusion (\ref{main result 2_ineq 1}).

In order to end the proof of (\ref{main result 2_ineq 1}) it is sufficient
to remark that, for any $0\leq s\leq t\leq T,$%
\begin{align*}
\left\vert x\left( t\right) -x\left( s\right) \right\vert ^{2} & \leq\left[
\mathbf{m}_{m}^{2}\left( t-s\right) +\mathbf{m}_{m}\left( t-s\right)
C_{T,m}+\left( t-s\right) \right] \cdot\exp\left[ C\left( 1+C_{T,m}\right) %
\right] \\
& \leq C_{T,m}~\boldsymbol{\upmu}_{m}\left( t-s\right) .
\end{align*}

\noindent\textrm{Step 5.}\textit{\ Getting inequality }(\ref{main result
2_ineq 2}).

Since $\left\updownarrow k\right\updownarrow _{T}+\updownarrow \hat{k}%
\updownarrow _{T}\leq C_{T,m}+C_{T,\hat{m}}$~, from inequality (\ref{uniq
ineq 2}) we have%
\begin{equation*}
\begin{array}{r}
\left\Vert x-\hat{x}\right\Vert _{T}^{2}\leq 2\big[\left\vert x_{0}-\hat{x}%
_{0}\right\vert ^{2}+\left\Vert m-\hat{m}\right\Vert _{T}^{2}+2\left\Vert m-%
\hat{m}\right\Vert _{T}\Vert k-\hat{k}\Vert _{T}\big] \\
\cdot \exp \big[4\gamma (2t+\left\updownarrow k\right\updownarrow
_{T}+\updownarrow \hat{k}\updownarrow _{T})\big]\smallskip \\
\multicolumn{1}{l}{\quad \quad \quad \quad \;\leq A^{2}\left( C_{T,m},C_{T,%
\hat{m}}\right) ~\left[ \left\vert x_{0}-\hat{x}_{0}\right\vert
^{2}+\left\Vert m-\hat{m}\right\Vert _{T}\right] ,}%
\end{array}%
\end{equation*}%
(where $A$ is a continuous function) and the conclusion follows since $k-%
\hat{k}=x_{0}-\hat{x}_{0}+m-\hat{m}-\left( x-\hat{x}\right) $.\hfill $%
\medskip $
\end{proof}

\begin{proof}[Proof of the Theorem \protect\ref{main result 1}]
Uniqueness was proved in Theorem \ref{uniq}. To prove existence, let $%
m_{n}\in C^{1}\left( \mathbb{R}_{+};\mathbb{R}^{d}\right) $ with $%
m_{n}\left( 0\right) =0$ be such that $\left\Vert m_{n}-m\right\Vert _{T}$ $%
\rightarrow 0$ for all $T\geq 0$. Since $m_{n}\in C^{1}\left( \mathbb{R}_{+};%
\mathbb{R}^{d}\right) $, we deduce, using the results from the papers \cite%
{de-ma-to/85} or \cite{ro-sa/06}, that there exists a unique solution $%
\left( x_{n},k_{n}\right) $ of the $\mathcal{SP}\left( \partial ^{-}\varphi
;x_{0},m_{n}\right) $, and by Corollary \ref{main result 3} we see that
there exist $x,k\in C\left( \mathbb{R}_{+};\mathbb{R}^{d}\right) $ such that
for all $T\geq 0$%
\begin{equation*}
\begin{array}{l}
\left\Vert x_{n}-x\right\Vert _{T}+\left\Vert k_{n}-k\right\Vert
_{T}\rightarrow 0,\;\text{as\ }n\rightarrow \infty \text{, and}\smallskip \\
\left( x,k\right) =\mathcal{SP}\left( \partial ^{-}\varphi ;x_{0},m\right) ,%
\end{array}%
\end{equation*}%
which completes the proof.\hfill
\end{proof}

\section{Generalized Skorohod equations}

Consider the next (non--convex) variational inequality with singular input
(which will be called generalized Skorohod differential equation):%
\begin{equation}
\left\{
\begin{array}{l}
x\left( t\right) +k\left( t\right) =x_{0}+\displaystyle\int_{0}^{t}f\left(
s,x\left( s\right) \right) ds+m\left( t\right) ,\quad t\geq0,\smallskip \\
dk\left( t\right) \in\partial^{-}\varphi\left( x\left( t\right) \right)
\left( dt\right)%
\end{array}
\right.  \label{GSE}
\end{equation}
(for the notation $dk\left( t\right) \in\partial^{-}\varphi\left( x\left(
t\right) \right) \left( dt\right) $ we recall Remark \ref%
{notation_apartenance}).

We introduce the following supplementary assumptions:%
\begin{equation}
f\left( \cdot ,x\right) :\mathbb{R}_{+}\rightarrow \mathbb{R}^{d}\text{ is
measurable, }\forall x\in \mathbb{R}^{d},  \label{assumpt f 1}
\end{equation}%
and there exists $\mu \in L_{loc}^{1}\left( \mathbb{R}_{+}\right) $, such
that a.e.\ $t\geq 0:$%
\begin{equation}
\begin{array}{rl}
\left( i\right) & x\mapsto f\left( t,x\right) :\mathbb{R}^{d}\rightarrow
\mathbb{R}^{d}\,\text{is continuous,}\smallskip \\
\left( ii\right) & \left\langle x-y,f(t,x)-f(t,y)\right\rangle \leq \mu
\left( t\right) |x-y|^{2},\;\forall x,y\in \mathbb{R}^{d},\smallskip \\
\left( iii\right) & \displaystyle\int_{0}^{T}f^{\#}\left( s\right) ds<\infty
,\;\forall T\geq 0,%
\end{array}
\label{assumpt f 2}
\end{equation}%
where%
\begin{equation}
f^{\#}\left( t\right) :=\sup \big\{\left\vert f(t,x)\right\vert :x\in
\overline{\mathrm{Dom}\left( \varphi \right) }\big\}.  \label{def f sharp}
\end{equation}%
Clearly, assumption (\ref{assumpt f 2}$-iii$) is satisfied if, as example, $%
f\left( t,x\right) =f\left( t\right) $ or if $\mathrm{Dom}\left( \varphi
\right) $ is bounded.

\begin{proposition}[Generalized Skorohod Equation]
\label{main result 5}Let $\varphi :\mathbb{R}^{d}\rightarrow(-\infty,\infty]$
and $f:\mathbb{R}_{+}\times\mathbb{R}^{d}\rightarrow\mathbb{R}^{d}$ be such
that assumptions (\ref{assumpt input}), (\ref{assumpt phi}), (\ref{assumpt
phi 2}), (\ref{assumpt phi 3}) and (\ref{assumpt f 1}), (\ref{assumpt f 2})
are satisfied.

Then the generalized Skorohod equation (\ref{GSE}) has a unique solution.
\end{proposition}

\begin{proof}
Let $\left( x,k\right) $ and $(\hat{x},\hat{k})$ be two solutions. Then
\begin{equation*}
\begin{array}{l}
|x\left( t\right) -\hat{x}\left( t\right) |^{2}+2\displaystyle%
\int_{0}^{t}\langle x\left( r\right) -\hat{x}\left( r\right) ,dk\left(
r\right) -d\hat{k}\left( r\right) \rangle\smallskip \\
=2\displaystyle\int_{0}^{t}\langle x\left( r\right) -\hat{x}\left( r\right)
,f\left( r,x\left( r\right) \right) -f\left( r,\hat{x}\left( r\right)
\right) \rangle dr\leq2\int_{0}^{t}\mu^{+}\left( r\right) |x\left( r\right) -%
\hat{x}\left( r\right) |^{2}dr,%
\end{array}%
\end{equation*}
and using Lemma \ref{uniq lemma} it follows that%
\begin{equation*}
\left\vert x\left( t\right) -\hat{x}\left( t\right) \right\vert
^{2}\leq2\int _{0}^{t}\left\vert x\left( r\right) -\hat{x}\left( r\right)
\right\vert ^{2}dA_{r},
\end{equation*}
where%
\begin{equation*}
A_{t}=2\rho t+\gamma\left\updownarrow k\right\updownarrow _{t}+\gamma
\updownarrow\hspace{-0.09cm}\hat{k}\hspace{-0.05cm}\updownarrow_{t}+\int
_{0}^{t}\mu^{+}\left( r\right) dr.
\end{equation*}
Applying a Gronwall's type inequality, we see that $x=\hat{x}$.$\smallskip$

Concerning the existence, we shall obtain the solution $\left( x,k\right) $
as the limit in $C\left( \left[ 0,T\right] ;\mathbb{R}^{d}\right) \times
C\left( \left[ 0,T\right] ;\mathbb{R}^{d}\right) $ of the sequence $\left(
x_{n},k_{n}\right) _{n\in\mathbb{N}^{\ast}}$ defined by an approximate
Skorohod equation%
\begin{equation}
\left\{
\begin{array}{l}
x_{n}\left( t\right) =x_{0}~,\;\text{for }t<0, \\
x_{n}\left( t\right) +k_{n}\left( t\right) =x_{0}+\displaystyle\int_{0}^{t}f%
\big(s,x_{n}(s-1/n)\big)ds+m\left( t\right) ,\quad\text{for }%
t\geq0,\smallskip \\
dk_{n}\left( t\right) \in\partial^{-}\varphi\left( x_{n}\left( t\right)
\right) \left( dt\right) .%
\end{array}
\right.  \label{GSE approx}
\end{equation}
For any $i\in\mathbb{N}$, for $t\in\left[ \frac{i}{n},\frac{i+1}{n}\right] $%
, we can write%
\begin{equation*}
x_{n}\left( t\right) +\left[ k_{n}\left( t\right) -k_{n}\left( i/n\right) %
\right] =x_{n}\left( i/n\right) +\int_{i/n}^{t}f\big(s,x_{n}(s-1/n)\big)%
ds+m\left( t\right) -m\left( i/n\right) ,
\end{equation*}
therefore by iteration over the intervals $\left[ \frac{i}{n},\frac{i+1}{n}%
\right] $ there exists (via Theorem \ref{main result 1}) a unique pair $%
\left( x_{n},k_{n}\right) =\mathcal{SP}\left( \partial^{-}\varphi
;x_{0},m_{n}\right) $, with%
\begin{equation*}
m_{n}\left( t\right) =\int_{0}^{t}f\big(s,x_{n}(s-1/n)\big)ds+m\left(
t\right) .
\end{equation*}
Let $T>0$. If $\mathcal{M}$ denotes the set%
\begin{equation*}
\mathcal{M}:=\left\{ m_{n}:n\in\mathbb{N}^{\ast}\right\} ,
\end{equation*}
then $\mathcal{M}$ is a relatively compact subset of $C\left( \left[ 0,T%
\right] ;\mathbb{R}^{d}\right) $ since it is a bounded and equicontinuous
subset of $C\left( \left[ 0,T\right] ;\mathbb{R}^{d}\right) $.

Indeed%
\begin{equation*}
\left\Vert m_{n}\right\Vert _{T}\leq \int_{0}^{T}f^{\#}\left( s\right)
ds+\left\Vert m\right\Vert _{T}
\end{equation*}%
and, for $s<t,$%
\begin{equation*}
\left\vert m_{n}\left( t\right) -m_{n}\left( s\right) \right\vert \leq
\int_{s}^{t}f^{\#}\left( r\right) dr+\left\vert m\left( t\right) -m\left(
s\right) \right\vert .
\end{equation*}%
Then by Theorem \ref{main result 2} and Remark \ref{remark 1},%
\begin{equation}
\left\Vert x_{n}\right\Vert _{T}+\left\updownarrow k_{n}\right\updownarrow
_{T}\leq \left\vert x_{0}\right\vert +C_{T,\mathcal{M}}  \label{bound}
\end{equation}%
and for all $0\leq s\leq t:$%
\begin{equation*}
\left\vert x_{n}\left( t\right) -x_{n}\left( s\right) \right\vert
+\left\updownarrow k_{n}\right\updownarrow _{t}-\left\updownarrow
k_{n}\right\updownarrow _{s}\leq C_{T,\mathcal{M}}\sqrt{\boldsymbol{\upmu}_{%
\mathcal{M}}\left( t-s\right) },
\end{equation*}%
where $\boldsymbol{\upmu}_{\mathcal{M}}\left( \varepsilon \right)
:=\varepsilon +\sup\limits_{m\in \mathcal{M}}\mathbf{m}_{m}.$

Hence, by Arzel\`{a}--Ascoli's theorem, the set $\left\{ x_{n}:n\in \mathbb{N%
}^{\ast}\right\} $ is a relatively compact subset of $C\left( \left[ 0,T%
\right] ;\mathbb{R}^{d}\right) $.

Let $x\in C\left( \left[ 0,T\right] ;\mathbb{R}^{d}\right) $ be such that,
along a sequence still denoted by $\left\{ x_{n}:n\in \mathbb{N}^{\ast
}\right\} $,%
\begin{equation*}
\left\Vert x_{n}-x\right\Vert _{T}\rightarrow 0,\;\text{as }n\rightarrow
\infty .
\end{equation*}%
Then, uniformly with respect to $t\in \left[ 0,T\right] ,$%
\begin{equation*}
m_{n}\left( t\right) \rightarrow \int_{0}^{t}f\left( s,x(s)\right)
ds+m\left( t\right) ,\;\text{as }n\rightarrow \infty ,
\end{equation*}%
and%
\begin{equation*}
k_{n}\left( t\right) \rightarrow k\left( t\right) =x_{0}+\int_{0}^{t}f\left(
s,x(s)\right) ds+m\left( t\right) -x\left( t\right) ,\;\text{as }%
n\rightarrow \infty .
\end{equation*}%
Using Corollary \ref{main result 3} we infer that%
\begin{equation*}
\left( x,k\right) =\mathcal{SP}\left( \partial ^{-}\varphi
;x_{0},\int_{0}^{\cdot }f\left( s,x(s)\right) ds+m\right)
\end{equation*}%
i.e. $\left( x,k\right) $ is a solution of problem (\ref{GSE}). The
uniqueness of the solution of implies that the whole sequence $\left(
x_{n},k_{n}\right) $ is convergent to that solution $\left( x,k\right) $.
The proof is completed now.\hfill $\smallskip $
\end{proof}

\begin{remark}
As it can be seen in the above proof, assumption (\ref{assumpt f 2}$-iii$)
is essential in order to obtain inequality (\ref{bound}), i.e. the
boundedness of the sequence $(x_{n},k_{n})_{n}$ defined by (\ref{GSE approx}%
).
\end{remark}

\begin{remark}
\label{remark 5}If we replace assumptions (\ref{assumpt phi 2}), (\ref%
{assumpt phi 3}) and (\ref{assumpt f 2}$-iii$) by the hypotheses%
\begin{equation*}
\mathrm{Int}\left( \mathrm{Dom}\left( \varphi \right) \right) \neq \emptyset
\end{equation*}%
and%
\begin{equation*}
\int_{0}^{T}f_{R}^{\#}\left( s\right) ds<\infty ,\quad \forall R,T>0,
\end{equation*}%
where%
\begin{equation*}
f_{R}^{\#}\left( t\right) :=\sup \big\{\left\vert f(t,x)\right\vert
:\left\vert x\right\vert \leq R\big\},
\end{equation*}%
then, following the calculus from the convex case (see, e.g., Remark 4.15
and Proposition 4.16 from \cite{pa-ra/14}), we deduce%
\begin{equation}
\begin{array}{l}
\displaystyle\frac{1}{2}\left\vert x\left( t\right) -u_{0}\right\vert ^{2}+%
\frac{r_{0}}{2}\left\updownarrow k\right\updownarrow _{t}+\frac{r_{0}}{2}%
\int_{0}^{t}\left\vert f\left( r,x\left( r\right) \right) \right\vert
dr\medskip  \\
\displaystyle\leq \frac{1}{4}\left\Vert x-u_{0}\right\Vert
_{t}^{2}+C_{0}+C_{0}\left\Vert m\right\Vert _{T}\medskip  \\
\displaystyle+2\int_{0}^{t}\mu ^{+}\left( r\right) \left\Vert
x-u_{0}\right\Vert _{r}^{2}dr+2\int_{0}^{t}\mu ^{+}\left( r\right)
\left\Vert x-u_{0}\right\Vert _{r}^{2}(\rho dr+\gamma d\left\updownarrow
k\right\updownarrow _{r}),%
\end{array}
\label{bound 2}
\end{equation}%
where $u_{0}\in \mathbb{R}^{d}$ and $r_{0}\in \lbrack 0,1)$ are such that $%
\overline{B}\left( u_{0},r_{0}\right) \subset \mathrm{Int}\left( \mathrm{Dom}%
\left( \varphi \right) \right) .$

In the convex case, namely $\gamma =0$, this inequality yields the
boundedness (\ref{main result 2_ineq 1}--$a,b$), but in the non--convex case
($\gamma \neq 0$) we cannot obtain (\ref{main result 2_ineq 1}--$a,b$).

Of course, if there exists $R>0$ such that $\mathrm{Dom}\left( \varphi
\right) \subset B\left( 0,R\right) $, then $\left\Vert x\right\Vert _{T}\leq
R$ and from (\ref{bound 2}) we obtain $\left\updownarrow k\right\updownarrow
_{T}\leq C$, if $\gamma $ is such that%
\begin{equation*}
0\leq \gamma <\frac{r_{0}}{2\left( r_{0}+R+\left\vert u_{0}\right\vert
\right) }\,.
\end{equation*}
\end{remark}

If in the above Proposition we take $\varphi =I_{E}$ we get, via Theorem \ref%
{main result 4},

\begin{corollary}[Skorohod equation]
Let $x_{0}\in E$ and $m:\mathbb{R}_{+}\rightarrow \mathbb{R}^{d}$ be a
continuous function such that $m\left( 0\right) =0$. If $f$ satisfies
assumption (\ref{assumpt f 1})--(\ref{assumpt f 2}) and $E$ satisfies the $%
r_{0}-UEBC$ (for some $r_{0}>0$) and the shifted uniform interior ball
condition, then there exists a unique pair $\left( x,k\right) $ such that:%
\begin{equation}
\begin{array}{rl}
\left( j\right) & x,k\in C\left( \mathbb{R}_{+};E\right) ,\text{\quad }%
k\left( 0\right) =0,\smallskip \\
\left( jj\right) & k\in BV_{loc}\left( \mathbb{R}_{+};\mathbb{R}^{d}\right)
,\smallskip \\
\left( jjj\right) & x\left( t\right) +k\left( t\right)
=x_{0}+\int_{0}^{t}f\left( s,x(s)\right) ds+m\left( t\right) ,\smallskip \\
\left( jv\right) & \left\updownarrow k\right\updownarrow _{t}=\int_{0}^{t}%
\mathbf{1}_{x\left( s\right) \in \mathrm{Bd}\left( E\right)
}d\left\updownarrow k\right\updownarrow _{s}~,\smallskip \\
\left( v\right) & \multicolumn{1}{r}{k\left( t\right)
=\int_{0}^{t}n_{x\left( s\right) }d\left\updownarrow k\right\updownarrow
_{s},\;\text{where }n_{x\left( s\right) }\in N_{E}\left( x\left( s\right)
\right) \quad \text{and}\smallskip} \\
& \multicolumn{1}{r}{\left\vert n_{x\left( s\right) }\right\vert
=1,\;d\left\updownarrow k\right\updownarrow _{s}\text{--a.e.}}%
\end{array}
\label{def sol GSE_indicator case}
\end{equation}
\end{corollary}

\section{Non--Convex stochastic variational inequalities}

In the last section of the paper we will study the following multivalued
stochastic differential equation (also called \textit{stochastic variational
inequality}) considered on a non--convex domain:%
\begin{equation}
\left\{
\begin{array}{l}
X_{t}+K_{t}=\xi+\displaystyle\int_{0}^{t}F\left( s,X_{s}\right)
ds+\int_{0}^{t}G\left( s,X_{s}\right) dB_{s},\quad t\geq0,\smallskip \\
dK_{t}\left( \omega\right) \in\partial^{-}\varphi\left( X_{t}\left(
\omega\right) \right) \left( dt\right) ,%
\end{array}
\right.  \label{GSSE}
\end{equation}
where $\varphi$ is $\left( \rho,\gamma\right) $--semiconvex function and $%
\left\{ B_{t}:t\geq0\right\} $ is an $\mathbb{R}^{k}$--valued Brownian
motion with respect to a stochastic basis (which is supposed to be complete
and right--continuous) $\left( \Omega,\mathcal{F},\mathbb{P},\{\mathcal{F}%
_{t}\}_{t\geq0}\right) $.

First we derive directly from Theorem \ref{main result 1} the existence
result in the additive noise case.

\begin{corollary}
\label{main result 1_conseq}Let $\left( \Omega,\mathcal{F},\mathbb{P},%
\mathcal{F}_{t},B_{t}\right) _{t\geq0}$ be given. If
\begin{equation*}
\xi\in L^{0}\big(\Omega,\mathcal{F}_{0},\mathbb{P};\overline{\mathrm{Dom}%
\left( \varphi\right) }\big)
\end{equation*}
and $M$ is a progressively measurable and continuous stochastic processes
(p.m.c.s.p. for short) with $M_{0}=0$, then there exists a unique pair $%
\left( X,K\right) $ of p.m.c.s.p., solution of the problem%
\begin{equation*}
\left\{
\begin{array}{l}
X_{t}\left( \omega\right) +K_{t}\left( \omega\right) =\xi\left(
\omega\right) +M_{t}\left( \omega\right) ,\;t\geq0,\;\omega\in
\Omega,\smallskip \\
dK_{t}\left( \omega\right) \in\partial^{-}\varphi\left( X_{t}\left(
\omega\right) \right) \left( dt\right)%
\end{array}
\right.
\end{equation*}
$\smallskip$(in this case we shall write $\left( X_{\cdot}\left(
\omega\right) ,K_{\cdot}\left( \omega\right) \right) =\mathcal{SP}\left(
\partial^{-}\varphi;\xi\left( \omega\right) ,M_{\cdot}\left( \omega\right)
\right) ,\quad\mathbb{P}-a.s.$).
\end{corollary}

\begin{proof}
Let $\omega $ be arbitrary but fixed. By Theorem \ref{main result 1}, the
Skorohod problem
\begin{equation*}
\left( X_{\cdot }\left( \omega \right) ,K_{\cdot }\left( \omega \right)
\right) =\mathcal{SP}\left( \partial ^{-}\varphi ;\xi \left( \omega \right)
,M_{\cdot }\left( \omega \right) \right)
\end{equation*}%
has a unique solution
\begin{equation*}
\left( X_{\cdot }\left( \omega \right) ,K_{\cdot }\left( \omega \right)
\right) \in C(\mathbb{R}_{+};\mathbb{R}^{d})\times C(\mathbb{R}_{+};\mathbb{R%
}^{d}).
\end{equation*}%
Since $\left( \omega ,t\right) \mapsto M_{t}\left( \omega \right) $ is
progressively measurable and the mapping%
\begin{equation*}
\left( \xi ,M\right) \mapsto X:\overline{\mathrm{Dom}\left( \varphi \right) }%
\times C(\left[ 0,t\right] ;\mathbb{R}^{d})\rightarrow C(\left[ 0,t\right] ;%
\mathbb{R}^{d})
\end{equation*}%
is continuous for each $0\leq t\leq T$, the stochastic process $X$ is
progressively measurable. Hence the conclusion follows.\hfill $\smallskip $
\end{proof}

The next assumptions will be needed throughout this section:

\begin{itemize}
\item[\textrm{(A}$_{1}$\textrm{)}] \textit{(Carath\'{e}odory conditions) The
functions }$F\left( \cdot,\cdot,\cdot\right) :\Omega\times\mathbb{R}%
_{+}\times\mathbb{R}^{d}\rightarrow\mathbb{R}^{d}$\textit{\ and }$G\left(
\cdot,\cdot,\cdot\right) :\Omega\times\mathbb{R}_{+}\times\mathbb{R}%
^{d}\rightarrow\mathbb{R}^{d\times k}$\textit{\ are }$\left( \mathcal{P},%
\mathbb{R}^{d}\right) $\textit{--Carath\'{e}odory functions, i.e.}%
\begin{equation}
\begin{array}{l}
F\left( \cdot,\cdot,x\right) \;\text{\textit{and}\ }G\left( \cdot
,\cdot,x\right) \mathit{\;}\text{\textit{are p.m.s.p.,}\ }\forall ~x\in%
\mathbb{R}^{d},\smallskip \\
F\left( \omega,t,\cdot\right) \;\text{\textit{and}\ }G\left( \omega
,t,\cdot\right) \;\text{\textit{are continuous function} }d\mathbb{P}\otimes
dt\text{\textit{--a.e.}}%
\end{array}
\label{assumpt Carath}
\end{equation}

\item[\textrm{(A}$_{2}$\textrm{)}] \textit{(Boundedness conditions)} \textit{%
For all }$T\geq0:$%
\begin{equation}
\displaystyle\int_{0}^{T}F^{\#}\left( s\right) ds<\infty\quad\text{and}\quad%
\displaystyle\int_{0}^{T}|G^{\#}\left( s\right) |^{2}ds<\infty ,\;\mathbb{P}%
\text{\textit{--a.s.,}}  \label{assumpt bound}
\end{equation}
\textit{where}%
\begin{align*}
F^{\#}\left( t\right) & :=\sup\big\{\left\vert F(t,x)\right\vert :x\in%
\overline{\mathrm{Dom}\left( \varphi\right) }\big\}, \\
G^{\#}\left( t\right) & :=\sup\big\{\left\vert G(t,x)\right\vert :x\in%
\overline{\mathrm{Dom}\left( \varphi\right) }\big\}
\end{align*}

\item[\textrm{(A}$_{3}$\textrm{)}] \textit{(Monotonicity and Lipschitz
conditions)} \textit{There exist }$\mu\in L_{loc}^{1}\left( \mathbb{R}%
_{+}\right) $ \textit{and} $\ell\in L_{loc}^{2}\left( \mathbb{R}_{+}\right) $
with $\ell\geq0,$ such that $d\mathbb{P}\otimes dt$--a.e.%
\begin{equation}
\begin{array}{rl}
\left( i\right) & \left\langle x-y,F(t,x)-F(t,y)\right\rangle \,\leq
\mu\left( t\right) |x-y|^{2},\quad\forall\,x,y\in\mathbb{R}^{d},\smallskip
\\
\left( ii\right) & |G(t,x)-G(t,y)|\leq\ell\left( t\right) |x-y|,\quad
\forall\,x,y\in\mathbb{R}^{d}\text{.}%
\end{array}
\label{assumpt lipsch}
\end{equation}
\end{itemize}

We define now the notions of strong and weak solutions for the stochastic
Skorohod equation (\ref{GSSE}).

\begin{definition}
Let $\left( \Omega,\mathcal{F},\mathbb{P},\mathcal{F}_{t},B_{t}\right)
_{t\geq0}$ be given. A pair $\left( X,K\right) :\Omega\times\mathbb{R}%
_{+}\rightarrow\mathbb{R}^{d}\times\mathbb{R}^{d}$ of continuous $\mathcal{F}%
_{t}-$p.m.c.s.p. is a strong solution of the stochastic Skorohod equation (%
\ref{GSSE}) if $\mathbb{P}$--a.s.$,$%
\begin{equation}
\begin{array}{rl}
\left( j\right) & X_{t}\in\overline{\mathrm{Dom}\left( \varphi\right) },%
\text{\ }\forall\,t\geq0,\;\varphi\left( X_{\cdot}\right) \in
L_{loc}^{1}\left( \mathbb{R}_{+}\right) ,\smallskip \\
\left( jj\right) & K_{\cdot}\in BV_{loc}\left( \mathbb{R}_{+};\mathbb{R}%
^{d}\right) ,\text{\quad}K_{0}=0,\smallskip \\
\left( jjj\right) & X_{t}+K_{t}=\xi+\displaystyle\int_{0}^{t}F\left(
s,X_{s}\right) ds+\displaystyle\int_{0}^{t}G\left( s,X_{s}\right)
dB_{s},\;\forall t\geq0,\smallskip \\
\left( jv\right) & \forall\,0\leq s\leq t,\;\forall y:\mathbb{R}%
_{+}\rightarrow\mathbb{R}^{d}\text{ continuous}\smallskip \\
& \displaystyle\int_{s}^{t}\left\langle y\left( r\right)
-X_{r},dK_{r}\right\rangle +\int_{s}^{t}\varphi\left( X_{r}\right)
dr\leq\int_{s}^{t}\varphi\left( y\left( r\right) \right) dr\smallskip \\
& \quad\quad\quad\quad\quad\quad\quad\quad+\displaystyle\int_{s}^{t}\left%
\vert y\left( r\right) -X_{r}\right\vert ^{2}\left( \rho dr+\gamma
d\left\updownarrow K\right\updownarrow _{r}\right) ,%
\end{array}
\label{def sol GSSE}
\end{equation}
which means that%
\begin{equation*}
\left( X_{\cdot}\left( \omega\right) ,K_{\cdot}\left( \omega\right) \right) =%
\mathcal{SP}\left( \partial^{-}\varphi;\xi\left( \omega\right)
,M_{\cdot}\left( \omega\right) \right) ,\quad\mathbb{P}-a.s.,
\end{equation*}
where%
\begin{equation*}
M_{t}=\int_{0}^{t}F\left( s,X_{s}\right) ds+\int_{0}^{t}G\left(
s,X_{s}\right) dB_{s}~.
\end{equation*}
\end{definition}

\begin{definition}
Let $F\left( \omega,t,x\right) :=f\left( t,x\right) $, $G\left(
\omega,t,x\right) :=g\left( t,x\right) $ and $\xi\left( \omega\right)
:=x_{0} $ (be independent of $\omega$). If there exists a stochastic basis $%
\left( \Omega,\mathcal{F},\mathbb{P},\mathcal{F}_{t}\right) _{t\geq0}$, an $%
\mathbb{R}^{k}$--valued $\mathcal{F}_{t}$--Brownian motion $\left\{
B_{t}:t\geq0\right\} $ and a pair $\left( X_{\cdot},K_{\cdot}\right)
:\Omega\times\mathbb{R}_{+}\rightarrow\mathbb{R}^{d}\times\mathbb{R}^{d}$ of
p.m.c.s.p. such that
\begin{equation*}
\left( X_{\cdot}\left( \omega\right) ,K_{\cdot}\left( \omega\right) \right) =%
\mathcal{SP}\left( \partial^{-}\varphi;x_{0},M_{\cdot}\left( \omega\right)
\right) ,\quad\mathbb{P}-a.s.,
\end{equation*}
where%
\begin{equation*}
M_{t}=\int _{0}^{t}f\left( s,X_{s}\right) ds+\int _{0}^{t}g\left(
s,X_{s}\right) dB_{s},
\end{equation*}
then the collection $\left( \Omega,\mathcal{F},\mathbb{P},\mathcal{F}%
_{t},B_{t},X_{t},K_{t}\right) _{t\geq0}$ it is called a weak solution of the
stochastic Skorohod equation (\ref{GSSE}).
\end{definition}

Since the stochastic process $K$ is uniquely determined from $\left(
X,B\right) $ through equation (\ref{def sol GSSE}$-jjj$), we can also say
that $X$ is a strong solution (and respectively $\left( \Omega,\mathcal{F},%
\mathbb{P},\mathcal{F}_{t},B_{t},X_{t}\right) _{t\geq0}$ is a weak solution).

We first give a uniqueness result for strong solutions.

\begin{proposition}
\label{uniq 2}(\textit{Pathwise uniqueness) }Let $\left( \Omega ,\mathcal{F},%
\mathbb{P},\mathcal{F}_{t},B_{t}\right) _{t\geq0}$ be given and assumptions (%
\ref{assumpt input}), (\ref{assumpt phi}) and (\ref{assumpt phi 2}) be
satisfied. The functions $F$ and $G$ are such that assumptions \textrm{(A}$%
_{1}$\textrm{--A}$_{3}$\textrm{)} are satisfied. Then the stochastic
Skorohod equation (\ref{GSSE}) has at most one strong solution.
\end{proposition}

\begin{proof}
Let $\left( X,K\right) $ and $(\hat{X},\hat{K})$ be two solutions
corresponding to $\xi$ and respectively $\hat{\xi}$. Since
\begin{equation*}
dK_{t}\in\partial^{-}\varphi\left( X_{t}\right) \left( dt\right) \quad\text{%
and}\quad d\hat{K}_{t}\in\partial^{-}\varphi(\hat{X}_{t})\left( dt\right) ,
\end{equation*}
by Lemma \ref{uniq lemma}, for $p\geq1$ and $\lambda>0$%
\begin{equation*}
\begin{array}{l}
\big\langle X_{t}-\hat{X}_{t},\left( F\left( t,X_{t}\right) dt-dK_{t}\right)
-(F(t,\hat{X}_{t})dt-d\hat{K}_{t})\big\rangle\smallskip \\
+\big(\dfrac{1}{2}m_{p}+9p\lambda\big)|G\left( t,X_{t}\right) -G(t,\hat {X}%
_{t})|^{2}dt\leq|X_{t}-\hat{X}_{t}|^{2}dV_{t},%
\end{array}%
\end{equation*}
where%
\begin{equation*}
V_{t}=\int_{0}^{t}\big[\mu\left( s\right) ds+\big(\dfrac{1}{2}m_{p}+9p\lambda%
\big)\ell^{2}\left( s\right) ds+2\rho ds+\gamma d\left\updownarrow
K\right\updownarrow _{s}+\gamma d\updownarrow \hspace{-0.09cm}\hat{K}\hspace{%
-0.05cm}\updownarrow_{s}\big].
\end{equation*}
Therefore, by Corollary \ref{technical result 5} (from the Appendix), we get%
\begin{equation*}
\mathbb{E}\left[ 1\wedge\big|\hspace{-0.05cm}\big|e^{-V}(X-\hat {X})\big|%
\hspace{-0.05cm}\big|_{T}^{p}\right] \leq C_{p,\lambda}~\mathbb{E}\left[
1\wedge|\xi-\hat{\xi}|^{p}\right] ,
\end{equation*}
and the uniqueness follows.\hfill$\smallskip$
\end{proof}

Remark also that in the case of additive noise (i.e. $G$ does not depend
upon $X$) we have existence of a strong solution.

\begin{lemma}
\label{main result 7}Let $\left( \Omega,\mathcal{F},\mathbb{P},\mathcal{F}%
_{t},B_{t}\right) _{t\geq0}$ be given and assumptions (\ref{assumpt input}),
(\ref{assumpt phi}), (\ref{assumpt phi 2}) and (\ref{assumpt phi 3}) be
satisfied. The function $F$ is such that assumptions \textrm{(A}$_{1}$%
\textrm{--A}$_{3}$\textrm{)} are satisfied.

If%
\begin{equation*}
\xi\in L^{0}\big(\Omega,\mathcal{F}_{0},\mathbb{P};\overline{\mathrm{Dom}%
\left( \varphi\right) }\big)
\end{equation*}
and $M$ is a p.m.c.s.p. with $M_{0}=0$, then there exists a unique pair $%
\left( X,K\right) $ of p.m.c.s.p., solution of the problem%
\begin{equation*}
\left( X_{\cdot}\left( \omega\right) ,K_{\cdot}\left( \omega\right) \right) =%
\mathcal{SP}\left( \partial^{-}\varphi;\xi\left( \omega\right)
,M_{\cdot}\left( \omega\right) \right) ,\quad\mathbb{P}-a.s.,
\end{equation*}
i.e.%
\begin{equation*}
\left\{
\begin{array}{l}
X_{t}\left( \omega\right) +K_{t}\left( \omega\right) =\xi\left(
\omega\right) +\displaystyle\int_{0}^{t}F\left( \omega,s,X_{s}\left(
\omega\right) \right) ds+M_{t}\left( \omega\right) ,\quad t\geq 0,\smallskip
\\
dK_{t}\left( \omega\right) \in\partial^{-}\varphi\left( X_{t}\left(
\omega\right) \right) \left( dt\right) ,%
\end{array}
\right.
\end{equation*}
$\mathbb{P}$--a.s.
\end{lemma}

\begin{proof}
Applying Corollary \ref{main result 1_conseq} to the approximating problem%
\begin{equation*}
\left\{
\begin{array}{l}
X_{t}^{n}\left( \omega\right) +K_{t}^{n}\left( \omega\right) =\xi\left(
\omega\right) +\displaystyle\int_{0}^{t}F(\omega,s,X_{s-1/n}^{n}\left(
\omega\right) )ds+M_{t}\left( \omega\right) ,\quad t\geq0,\smallskip \\
dK_{t}^{n}\left( \omega\right) \in\partial^{-}\varphi\left( X_{t}^{n}\left(
\omega\right) \right) \left( dt\right) ,%
\end{array}
\right.
\end{equation*}
we conclude that there exists a unique solution $\left( X^{n},K^{n}\right) $
of p.m.c.s.p. The solution $\left( X,K\right) $ is obtained as the limit of
the sequence $\left( X^{n},K^{n}\right) $, exactly as in the proof of
Proposition \ref{main result 5}.\hfill$\smallskip$
\end{proof}

In order to study the general stochastic Skorohod equation (\ref{GSSE}) we
shall consider only the case when $F,G$ and $\xi$ are independent of $\omega$
and, to highlight this, the coefficients will be denoted by $f$ and $g$
respectively.

Let us consider equation%
\begin{equation}
\left\{
\begin{array}{l}
X_{t}+K_{t}=x_{0}+\displaystyle\int_{0}^{t}f\left( s,X_{s}\right)
ds+\int_{0}^{t}g\left( s,X_{s}\right) dB_{s},\quad t\geq0,\smallskip \\
dK_{t}\left( \omega\right) \in\partial^{-}\varphi\left( X_{t}\left(
\omega\right) \right) \left( dt\right) ,%
\end{array}
\right.  \label{GSSE 2}
\end{equation}
where $f:\mathbb{R}_{+}\times\mathbb{R}^{d}\rightarrow\mathbb{R}^{d}$ and $g:%
\mathbb{R}_{+}\times\mathbb{R}^{d}\rightarrow\mathbb{R}^{d\times k}$.

We recall the definition of $f^{\#}$, $g^{\#}$ given as in (\ref{def f sharp}%
).

\begin{theorem}
\label{main result 6}Let assumptions (\ref{assumpt input}), (\ref{assumpt
phi}), (\ref{assumpt phi 2}) and (\ref{assumpt phi 3}) be satisfied. The
functions $f$ and $g$ are suppose to be $\left( \mathcal{B}_{1},\mathbb{R}%
^{d}\right) $--Carath\'{e}odory functions satisfying moreover the
boundedness conditions%
\begin{equation*}
\int _{0}^{T}\left[ |f^{\#}\left( s\right) |^{2}+|g^{\#}\left( s\right) |^{4}%
\right] ds<\infty,\;\forall T\geq0.
\end{equation*}
If $x_{0}\in\overline{\mathrm{Dom}\left( \varphi\right) }$ then equation (%
\ref{GSSE 2}) has a weak solution $(\Omega,\mathcal{F},\mathbb{P},\mathcal{F}%
_{t},X_{t},K_{t},$ $B_{t})_{t\geq0}$ .
\end{theorem}

\begin{remark}
Usually, when $G$ is Lipschitz, a fixed point argument is used (based on
Banach contraction theorem). But, in our case this argument doesn't work
even for the drift part $F\equiv0$, as it can be see from inequality (\ref%
{main result 2_ineq 2}), we have different order of the estimates in the
left and in the right side of the inequality.
\end{remark}

\begin{proof}
\noindent\textrm{Step 1.}\textit{\ Approximating sequence.}

Let $\left( \Omega,\mathcal{F},\mathbb{P},\mathcal{F}_{t}^{B},B_{t}\right)
_{t\geq0}$ be a stochastic basis. Applying Lemma \ref{main result 7}, we
deduce that there exists a unique pair $\left( X^{n},K^{n}\right)
:\Omega\times\mathbb{R}_{+}\rightarrow\mathbb{R}^{d}\times\mathbb{R}^{d}$ of
$\mathcal{F}_{t}^{B}$--p.m.c.s.p. such that%
\begin{equation}
\left\{
\begin{array}{l}
X_{t}^{n}+K_{t}^{n}=x_{0}+\displaystyle\int_{0}^{t}f(s,X_{s-1/n}^{n})ds+%
\int_{0}^{t}g(s,X_{s-1/n}^{n})dB_{s},\quad t\geq0,\smallskip \\
dK_{t}^{n}\left( \omega\right) \in\partial^{-}\varphi\left( X_{t}^{n}\left(
\omega\right) \right) \left( dt\right) .%
\end{array}
\right.  \label{GSSE approx}
\end{equation}
Denote%
\begin{equation*}
M_{t}^{n}:=\int_{0}^{t}f\left( s,X_{s-1/n}^{n}\right) ds+\int_{0}^{t}g\left(
s,X_{s-1/n}^{n}\right) dB_{s}~.
\end{equation*}
Since%
\begin{equation*}
\begin{array}{l}
\mathbb{E}\left[ \sup\limits_{0\leq\theta\leq\varepsilon}\left\vert
M_{t+\theta}^{n}-M_{t}^{n}\right\vert ^{4}\right] \leq C\left[ \left( %
\displaystyle\int_{t}^{t+\varepsilon}f^{\#}\left( s\right) ds\right)
^{4}+\left( \int_{t}^{t+\varepsilon}\left\vert g^{\#}\left( s\right)
\right\vert ^{2}ds\right) ^{2}\right] \smallskip \\
\quad\leq\varepsilon C\left[ \sup_{t\in\left[ 0,T\right] }\left( %
\displaystyle\int_{t}^{t+\varepsilon}\left\vert f^{\#}\left( s\right)
\right\vert ^{2}ds\right) ^{2}+\sup_{t\in\left[ 0,T\right]
}\int_{t}^{t+\varepsilon}\left\vert g^{\#}\left( s\right) \right\vert ^{4}ds%
\right]%
\end{array}%
\end{equation*}
using Proposition \ref{technical result 4}, we deduce that the family of
laws of $\left\{ M^{n}:n\geq1\right\} $ is tight on $C\left( \mathbb{R}_{+};%
\mathbb{R}^{d}\right) .$

Therefore by Theorem \ref{technical result 3} for all $T\geq0$%
\begin{equation*}
\lim_{N\nearrow\infty}\left[ \sup_{n\geq1}\mathbb{P}\left( \left\Vert
M^{n}\right\Vert _{T}\geq N\right) \right] =0,
\end{equation*}
and, for all $a>0$ and $T>0,$%
\begin{equation}
\lim_{\varepsilon\searrow0}\left[ \sup_{n\geq1}\mathbb{P}\big(\left\{
\mathbf{m}_{M^{n}}\left( \varepsilon\right) \geq a\right\} \big)\right] =0.
\label{technical result 2}
\end{equation}
Recalling the definition%
\begin{equation*}
\boldsymbol{\upmu}_{M^{n}}=\varepsilon+\mathbf{m}_{M^{n}}\left(
\varepsilon\right) ,
\end{equation*}
where $\mathbf{m}$ is the modulus of continuity, we see that we can replace
in \ref{technical result 2} $\mathbf{m}_{M^{n}}$ by $\boldsymbol{\upmu}%
_{M^{n}}.\smallskip$

\noindent\textrm{Step 2}\textit{\ Tightness.}

Let $T\geq0$ be arbitrary. We now show that the family of laws of the random
variables $U^{n}=\left( X^{n},K^{n},\left\updownarrow
K^{n}\right\updownarrow \right) $ is tight on $C\left( \left[ 0,T\right] ;%
\mathbb{R}^{2d+1}\right) .$

From (\ref{main result 2_ineq 1}$-c$) we deduce that%
\begin{equation*}
\mathbf{m}_{U^{n}}\left( \varepsilon\right) \leq G\left( M^{n}\right) \sqrt{%
\boldsymbol{\upmu}_{M^{n}}\left( \varepsilon\right) },\;\text{a.s.,}
\end{equation*}
where $G:C\left( \left[ 0,T\right] ;\mathbb{R}^{d}\right) \rightarrow
\mathbb{R}_{+}$ and%
\begin{align*}
G\left( x\right) & :=C_{T,x}=\exp\left[ C\left( 1+T+\left\Vert x\right\Vert
_{T}+B_{x}\right) \right] , \\
B_{x} & :=1/\boldsymbol{\upmu}_{x}^{-1}\left( \delta^{2}e^{-C\left(
1+T+\left\Vert x\right\Vert _{T}\right) }\right) .
\end{align*}
From (\ref{technical result 1}) we see that $G$ is bounded on compact subset
of $C\left( \left[ 0,T\right] ;\mathbb{R}^{d}\right) $ and therefore by
Proposition \ref{technical result 7}, $\left\{ U^{n};n\in\mathbb{N}^{\ast
}\right\} $ is tight on $C\left( \left[ 0,T\right] ;\mathbb{R}^{d}\right) $.

Using the Prohorov theorem we see that there exists a subsequence (still
denoted with $n$) such that%
\begin{equation*}
\left( X^{n},K^{n},\left\updownarrow K^{n}\right\updownarrow ,B\right)
\rightarrow\left( X,K,V,B\right) \quad\text{in law,\ as }n\rightarrow\infty
\end{equation*}
on $C\left( \left[ 0,T\right] ;\mathbb{R}^{2d+1+k}\right) $ and applying the
Skorohod theorem, we can choose a probability space $\left( \Omega,\mathcal{F%
},\mathbb{P}\right) $ and some random quadruples $(\bar {X}^{n},\bar{K}^{n},%
\bar{V}^{n},\bar{B}^{n})$, $(\bar{X},\bar{K},\bar{V},\bar{B})$ defined on $%
\left( \Omega,\mathcal{F},\mathbb{P}\right) $ such that%
\begin{align*}
\mathcal{L}(\bar{X}^{n},\bar{K}^{n},\bar{V}^{n},\bar{B}^{n}) & =\mathcal{L}%
\left( X^{n},K^{n},\left\updownarrow K^{n}\right\updownarrow ,B^{n}\right) \\
\mathcal{L}(\bar{X},\bar{K},\bar{V},\bar{B}) & =\mathcal{L}(X,K,V,B)
\end{align*}
and%
\begin{equation*}
(\bar{X}^{n},\bar{K}^{n},\bar{V}^{n},\bar{B}^{n})%
%TCIMACRO{%
%\TeXButton{\xrightarrow[]{\mathbb{P}-a.s.}}{\xrightarrow[]{\mathbb
%{P}-a.s.}}}%
%BeginExpansion
\xrightarrow[]{\mathbb
{P}-a.s.}%
%EndExpansion
(\bar{X},\bar{K},\bar{V},\bar{B}),\;\text{as }n\rightarrow\infty,\;\text{in }%
C(\left[ 0,T\right] ;\mathbb{R}^{2d+1+k}).
\end{equation*}
From Proposition \ref{technical result 8} we deduce that $%
%TCIMACRO{\TeXButton{\big(}{\big(}}%
%BeginExpansion
\big(%
%EndExpansion
\bar{B}^{n},\{\mathcal{F}_{t}^{\bar{X}^{n},\bar{K}^{n},\bar{V}^{n},\bar{B}%
^{n}}\}%
%TCIMACRO{\TeXButton{\big)}{\big)}}%
%BeginExpansion
\big)%
%EndExpansion
$, $n\geq1$, and $%
%TCIMACRO{\TeXButton{\big(}{\big(}}%
%BeginExpansion
\big(%
%EndExpansion
\bar{B},\{\mathcal{F}_{t}^{\bar{X},\bar{K},\bar{V},\bar{B}}\}%
%TCIMACRO{\TeXButton{\big)}{\big)}}%
%BeginExpansion
\big)%
%EndExpansion
$ are Brownian motions.$\smallskip$

\noindent\textrm{Step 3}\textit{\ Passing to the limit.}

Since we also have $\left( X^{n},K^{n},\left\updownarrow
K^{n}\right\updownarrow ,B\right) \rightarrow(\bar{X},\bar{K},\bar{V},\bar{B}%
),$ in law, then by Corollary \ref{technical result 9} we deduce that for
all $0\leq s\leq t$, $\mathbb{P}$--a.s.%
\begin{equation}
\begin{array}{l}
\bar{X}_{0}=x_{0},\quad\bar{K}_{0}=0,\quad\bar{X}_{t}\in\overline {\mathrm{%
Dom}\left( \varphi\right) },\smallskip \\
\left\updownarrow \bar{K}\right\updownarrow _{t}-\left\updownarrow \bar {K}%
\right\updownarrow _{s}\leq\bar{V}_{t}-\bar{V}_{s}\quad\text{and}\quad0=\bar{%
V}_{0}\leq\bar{V}_{s}\leq\bar{V}_{t}%
\end{array}
\label{technical ineq 5}
\end{equation}
Moreover, since for all $0\leq s<t$,$\;n\in\mathbb{N}^{\ast}$%
\begin{equation*}
\begin{array}{r}
\displaystyle\int_{s}^{t}\varphi\left( X_{r}^{n}\right)
dr\leq\int_{s}^{t}\varphi\left( y\left( r\right) \right)
dr-\int_{s}^{t}\left\langle y\left( r\right)
-X_{r}^{n},dK_{r}^{n}\right\rangle \smallskip \\
\displaystyle+\int_{s}^{t}\left\vert y\left( r\right) -X_{r}^{n}\right\vert
^{2}\left( \rho dr+\gamma d\left\updownarrow K^{n}\right\updownarrow
_{r}\right) ,\;\text{a.s.,}%
\end{array}%
\end{equation*}
then by Corollary \ref{technical result 10} we infer that%
\begin{equation}
\begin{array}{r}
\displaystyle\int_{s}^{t}\varphi\left( \bar{X}_{r}\right)
dr\leq\int_{s}^{t}\varphi\left( y\left( r\right) \right)
dr-\int_{s}^{t}\left\langle y\left( r\right) -\bar{X}_{r},d\bar{K}%
_{r}\right\rangle \smallskip \\
+\displaystyle\int_{s}^{t}\left\vert y\left( r\right) -\bar{X}%
_{r}\right\vert ^{2}\left( \rho dr+\gamma d\bar{V}_{r}\right) .%
\end{array}
\label{technical ineq 6}
\end{equation}
Hence, based on (\ref{technical ineq 5}) and (\ref{technical ineq 7}) and
Lemma \ref{echiv for (jv)} we have
\begin{equation*}
d\bar{K}_{r}\in\partial^{-}\varphi\left( \bar{X}_{r}\right) \left( dr\right)
.
\end{equation*}
Let%
\begin{equation*}
S_{t}\left( Y,B\right) :=x_{0}+\int _{0}^{t}f\left( s,Y_{s}\right) ds+\int
_{0}^{t}g\left( s,Y_{s}\right) dB_{s},\;\;t\geq0.
\end{equation*}
By Proposition \ref{technical result 11} it follows%
\begin{equation*}
\mathcal{L}(\bar{X}^{n},\bar{K}^{n},\bar{V}^{n},\bar{B}^{n},S_{t}(\bar{X}%
^{n},\bar{B}^{n}))=\mathcal{L}\left( X^{n},K^{n},\left\updownarrow
K^{n}\right\updownarrow ,B^{n},S_{t}(X^{n},B^{n})\right)
\end{equation*}
Since for every $t\geq0,$%
\begin{equation*}
X_{t}^{n}+K_{t}^{n}-S_{t}(X^{n},B^{n})=0,\;a.s.,
\end{equation*}
then by Corollary \ref{technical result 9} we have%
\begin{equation*}
\bar{X}_{t}^{n}+\bar{K}_{t}^{n}-S_{t}(\bar{X}^{n},\bar{B}^{n})=0,\;a.s.,
\end{equation*}
and consequently, letting $n\rightarrow\infty$,%
\begin{equation*}
\bar{X}_{t}+\bar{K}_{t}-S_{t}(\bar{X},\bar{B})=0,\;a.s.
\end{equation*}
Hence we obtain that, $\mathbb{P}$--a.s.,%
\begin{equation*}
\bar{X}_{t}+\bar{K}_{t}=x_{0}+\int _{0}^{t}f\left( s,\bar{X}_{s}\right)
ds+\int _{0}^{t}g\left( s,\bar{X}_{s}\right) d\bar{B}_{s},\;\forall~t\in%
\left[ 0,T\right] ,
\end{equation*}
and consequently $\Big(\bar{\Omega},\mathcal{\bar{F}},\mathbb{\bar{P}},%
\mathcal{F}_{t}^{\bar{B},\bar{X}},\bar{X}_{t},\bar{K}_{t},\bar{B}_{t}\Big)%
_{t\geq0}$ is a weak solution.\hfill$\smallskip$
\end{proof}

Since the stochastic process $K$ is uniquely determined by $\left(
X,B\right) $ via equation (\ref{GSSE 2}), then a weak solution for the
stochastic differential equation is a sextuplet ($\Omega $,$\mathcal{F}$,$%
\mathbb{P}$,$\left\{ \mathcal{F}_{t}\right\} _{t\geq 0}$,$X$,$B$). We know
that weak existence and pathwise uniqueness\textit{\ }implies strong
existence (see Theorem 3.55 in \cite{pa-ra/14} or Theorem 1.1 in \cite%
{ik-wa/81}). Hence we deduce from Theorem \ref{main result 6} and
Proposition \ref{uniq 2}:

\begin{theorem}
Let assumptions (\ref{assumpt input}), (\ref{assumpt phi}), (\ref{assumpt
phi 2}) and (\ref{assumpt phi 3}) be satisfied. The functions $f$ and $g$
are suppose to be $\left( \mathcal{B}_{1},\mathbb{R}^{d}\right) $--Carath%
\'{e}odory functions satisfying moreover assumption \textrm{(A}$_{3}$\textrm{%
)} and boundedness conditions%
\begin{equation*}
\int _{0}^{T}\left[ |f^{\#}\left( s\right) |^{2}+|g^{\#}\left( s\right) |^{4}%
\right] ds<\infty,\;\forall T\geq0.
\end{equation*}
If $x_{0}\in\overline{\mathrm{Dom}\left( \varphi\right) }$ then equation (%
\ref{GSSE 2}) has a unique strong solution $\left( X_{t},K_{t}\right)
_{t\geq0}~.$
\end{theorem}

\section{Appendix}

\subsection{Applications of Fatou's Lemma}

The next result is a well known consequence of weak convergence of
probability measures (for its proof see, e.g. \cite[Proposition 1.22]%
{pa-ra/14}).

\begin{proposition}
\label{technical result 12}Let $\left( \mathbb{X},\rho\right) $ be a
separable metric space. Let $\varphi:\mathbb{X}\rightarrow(-\infty,\infty]$
be a lower semicontinuous function. If $X$ and $X_{n}$ are $\mathbb{X}$%
--valued random variable, for $n\in\mathbb{N}^{\ast},$ such that%
\begin{equation*}
\begin{array}{rl}
\left( i\right) & X_{n}%
%TCIMACRO{\TeXButton{\xrightarrow[]{law}}{\xrightarrow[]{law}}}%
%BeginExpansion
\xrightarrow[]{law}%
%EndExpansion
X,\;\text{as }n\rightarrow\infty,%
\end{array}%
\end{equation*}
and there exists a continuous function $\alpha:\mathbb{X}\rightarrow \mathbb{%
R}$ such that
\begin{equation*}
\begin{array}{rl}
\left( ii\right) & \alpha\left( x\right) \leq\varphi\left( x\right)
,\quad\forall~x\in\mathbb{X},\smallskip \\
\left( iii\right) & \left\{ \alpha\left( X_{n}\right) :n\in \mathbb{N}%
^{\ast}\right\} \text{ is a uniformly integrable family,}%
\end{array}%
\end{equation*}
then the expectations $\mathbb{E}\varphi\left( X\right) $ and $\mathbb{E}%
\varphi\left( X_{n}\right) $ exist for all $n\in\mathbb{N},$ and
\begin{equation*}
\mathbb{-}\infty<\mathbb{E}\varphi\left( X\right) \leq\liminf_{n\rightarrow
+\infty}\mathbb{E}\varphi\left( X_{n}\right) ~.
\end{equation*}
\end{proposition}

For $0\leq s<t\leq T$, we denote by $\left\updownarrow X\right\updownarrow _{%
\left[ s,t\right] }$ (similar to (\ref{def total var})) the total variation
of $X_{\cdot}$ on $\left[ s,t\right] $, that is%
\begin{equation*}
\left\updownarrow X\right\updownarrow _{\left[ s,t\right] }\left(
\omega\right) =\sup\Big\{\sum \limits_{i=0}^{n-1}\left\vert
X_{t_{i+1}}\left( \omega\right) -X_{t_{i}}\left( \omega\right) \right\vert
:n\in\mathbb{N}^{\ast},s=t_{0}<t_{1}<\cdots<t_{n}=t\Big\}
\end{equation*}
We also use $\left\updownarrow X\right\updownarrow _{T}:=\left\updownarrow
X\right\updownarrow _{\left[ 0,T\right] }~.\smallskip$

Applying successively Proposition \ref{technical result 12} for%
\begin{align*}
\varphi \left( x\right) & =\mathrm{d}_{F}\left( x\left( t\right) \right) , \\
\varphi \left( x,y\right) & =\left( \sum_{i=0}^{N-1}\left\vert x\left(
t_{i+1}\right) -x\left( t_{i}\right) \right\vert -g\left( y\right) \right)
^{+},
\end{align*}%
where $s=t_{0}<t_{1}<\ldots <t_{N}=t$ is an arbitrary partition of $\left[
s.t\right] ,$ and respectively%
\begin{equation*}
\varphi \left( x\right) =\left( x\left( s\right) -x\left( t\right) \right)
^{+},
\end{equation*}%
we get the next result:

\begin{corollary}
\label{technical result 9}Let $s,t$ be arbitrary fixed such that $0\leq
s\leq t\leq T.$ If $g:C\left( \left[ 0,T\right] ;\mathbb{R}^{d}\right)
\rightarrow \mathbb{R}_{+}$ is a continuous function and $X$,$V$,$X^{n}$,$%
V^{n}$, $n\in \mathbb{N}^{\ast }$, are random variables with values in $%
C\left( \left[ 0,T\right] ;\mathbb{R}^{d}\right) $, such that%
\begin{equation*}
\left( X^{n},V^{n}\right)
%TCIMACRO{\TeXButton{\xrightarrow[]{law}}{\xrightarrow[]{law}}}%
%BeginExpansion
\xrightarrow[]{law}%
%EndExpansion
\left( X,V\right) ,\;\text{as }n\rightarrow \infty ,
\end{equation*}%
then the following implications hold true:

\begin{itemize}
\item[$\left( a\right) $] If $X_{t}^{n}\in F$ a.s., then $X_{t}\in F,$ a.s.,
whenever $F$ is closed subset of $\mathbb{R}^{d}$;

\item[$\left( b\right) $] If $\left\updownarrow X^{n}\right\updownarrow _{%
\left[ s,t\right] }\leq g\left( V^{n}\right) $ a.s., then $\left\updownarrow
X\right\updownarrow _{\left[ s,t\right] }\leq g\left( V\right) $, a.s.;

\item[$\left( c\right) $] If $d=1$ and $X_{s}^{n}\leq X_{t}^{n}$ a.s., then $%
X_{s}\leq X_{t}$, a.s.$\smallskip $
\end{itemize}
\end{corollary}

Now let us consider the partition%
\begin{equation*}
\Delta_{N}:s=r_{0}<r_{1}<\ldots<r_{N}=t,\;\;r_{i+1}-r_{i}=\frac{t-s}{N}
\end{equation*}
and the function $g:C\left( \left[ 0,T\right] ;\mathbb{R}^{d}\right)
\rightarrow\left[ 0,1\right] $, defined by%
\begin{equation*}
\begin{array}{l}
g\left( x,k,v\right) :=\left( \sum_{i=0}^{N-1}\left\vert k\left(
r_{i+1}\right) -k\left( r_{i}\right) \right\vert -v\left( t\right) +v\left(
s\right) \right) ^{+}\wedge1\smallskip \\
+\left[ \displaystyle\int_{s}^{t}\varphi\left( x\left( r\right) \right)
dr-\sum_{i=0}^{N-1}\left\langle x\left( r_{i}\right) ,k\left( r_{i+1}\right)
-k\left( r_{i}\right) \right\rangle -\mathbf{m}_{x}\left( 1/N\right) \left(
v\left( t\right) -v\left( s\right) \right) \right] ^{+}\wedge1.%
\end{array}%
\end{equation*}

Applying again the generalization of the Fatou's Lemma (Proposition \ref%
{technical result 12}), it can be proved:

\begin{corollary}
\label{technical result 10} Let $\left( X,K,V\right) $, $\left(
X^{n},K^{n},V^{n}\right) $, $n\in\mathbb{N}$, be $C\left( \left[ 0,T\right] ;%
\mathbb{R}^{d}\right) ^{2}\times C\left( \left[ 0,T\right] ;\mathbb{R}%
\right) $--valued random variables, such that%
\begin{equation*}
\left( X^{n},K^{n},V^{n}\right)
%TCIMACRO{%
%\TeXButton{\xrightarrow[n\rightarrow\infty]{law}}{\xrightarrow
%[n\rightarrow\infty]{law}}}%
%BeginExpansion
\xrightarrow
[n\rightarrow\infty]{law}%
%EndExpansion
\left( X,K,V\right)
\end{equation*}
and for all $0\leq s<t,$ and $n\in\mathbb{N}^{\ast},$%
\begin{equation*}
\left\updownarrow K^{n}\right\updownarrow _{t}-\left\updownarrow
K^{n}\right\updownarrow _{s}\leq V_{t}^{n}-V_{s}^{n}\;\;a.s.~.
\end{equation*}
If $\varphi:\mathbb{R}^{d}\rightarrow(-\infty,+\infty]$ is a lower
semicontinuous function and%
\begin{equation*}
\int _{s}^{t}\varphi\left( X_{r}^{n}\right) dr\leq\int _{s}^{t}\left\langle
X_{r}^{n},dK_{r}^{n}\right\rangle \text{, a.s. for all }n\in\mathbb{N}%
^{\ast},
\end{equation*}
then%
\begin{equation*}
\left\updownarrow K\right\updownarrow _{t}-\left\updownarrow
K\right\updownarrow _{s}\leq V_{t}-V_{s}\text{, a.s.}
\end{equation*}
and%
\begin{equation*}
\int _{s}^{t}\varphi\left( X_{r}\right) dr\leq\int _{s}^{t}\left\langle
X_{r}~,dK_{r}\right\rangle \text{, a.s.}
\end{equation*}
\end{corollary}

\subsection{Complements on tightness}

If $\left\{ X_{t}^{n}:t\geq0\right\} $, $n\in\mathbb{N}^{\ast},$ is a family
of continuous stochastic processes then the following result is a
consequence of the Arzel\`{a}--Ascoli theorem (see, e.g., Theorem 7.3 in
\cite{bi/99}).

We recall the notations:%
\begin{equation*}
\begin{array}{l}
\left\Vert X^{n}\right\Vert _{T}:=\sup\left\{ \left\vert
X_{t}^{n}\right\vert :t\in\left[ 0,T\right] \right\} ,\smallskip \\
\mathbf{m}_{X^{n}}\left( \varepsilon;\left[ 0,T\right] \right) :=\sup\left\{
\left\vert X_{t}^{n}-X_{s}^{n}\right\vert :t,s\in\left[ 0,T\right]
,\;\left\vert t-s\right\vert \leq\varepsilon\right\} .%
\end{array}%
\end{equation*}

\begin{theorem}
\label{technical result 3}The family $\left\{ X^{n}:n\in\mathbb{N}^{\ast
}\right\} $\textit{\ is tight in} $C(\mathbb{R}_{+};\mathbb{R}^{d})$ \textit{%
if and only if, for every }$T\geq0,$%
\begin{equation*}
\begin{array}{rl}
\left( j\right) & \lim\limits_{N\nearrow\infty}\left[ \sup\limits_{n\geq 1}%
\mathbb{P}^{\left( n\right) }\left( \left\vert X_{0}^{n}\right\vert \geq
N\right) \right] =0,\smallskip \\
\left( jj\right) & \lim\limits_{\varepsilon\searrow0}\left[ \sup
\limits_{n\geq1}\mathbb{P}^{\left( n\right) }\left( \mathbf{m}_{X^{n}}\left(
\varepsilon;\left[ 0,T\right] \right) \geq a\right) \right] =0,\quad\forall
a>0.%
\end{array}%
\end{equation*}
\textit{Moreover, tightness yields that for all} $T>0$%
\begin{equation*}
\lim\limits_{N\nearrow\infty}\left[ \sup\limits_{n\geq1}\mathbb{P}^{\left(
n\right) }\left( \left\Vert X^{n}\right\Vert _{T}\geq N\right) \right] =0.
\end{equation*}
\end{theorem}

Without using the above theorem, it can be proved the following criterion
for tightness which is well adapted to our needs. The proof can be found in
E. Pardoux and A. R\u{a}\c{s}canu \cite{pa-ra/14} (Proposition 1.53) and we
will give the sketch of the proof.

\begin{proposition}
\label{technical result 4}Let $\left\{ X_{t}^{n}:t\geq0\right\} $, $n\in%
\mathbb{N}^{\ast}$, be a family of $\mathbb{R}^{d}$--valued continuous
stochastic processes defined on probability space $\left( \Omega ,\mathcal{F}%
,\mathbb{P}\right) .$ Suppose that for every $T\geq0$, there exist $\alpha>0$
and $b\in C\left( \mathbb{R}_{+}\right) $ with $b(0)=0$, such that%
\begin{equation*}
\begin{array}{rl}
\left( j\right) & \lim\limits_{N\rightarrow\infty}\left[ \sup \limits_{n\in%
\mathbb{N}^{\ast}}\mathbb{P}(\{\left\vert X_{0}^{n}\right\vert \geq N\})%
\right] =0,\smallskip \\
\left( jj\right) & \mathbb{E}\left[ 1\wedge\sup\limits_{0\leq
s\leq\varepsilon}\left\vert X_{t+s}^{n}-X_{t}^{n}\right\vert ^{\alpha}\right]
\leq\varepsilon\cdot b(\varepsilon),\;\forall\varepsilon>0\text{, }%
n\geq1,\;t\in\left[ 0,T\right] .%
\end{array}%
\end{equation*}
Then the family $\left\{ X^{n}:n\in\mathbb{N}^{\ast}\right\} $ is tight in $%
C(\mathbb{R}_{+};\mathbb{R}^{d}).$
\end{proposition}

\begin{proof}
We fix $\varepsilon,T>0$. From $\left( j\right) $, there exists $%
M=M_{\varepsilon}\geq1$ such that
\begin{equation*}
\sup\limits_{n\in\mathbb{N}^{\ast}}\mathbb{P}(\{\left\vert
X_{0}^{n}\right\vert \geq M\})<\frac{\varepsilon}{2}.
\end{equation*}
Let $\gamma_{k}=\dfrac{1}{2^{\left( k-1\right) /\alpha}}$ and $\varepsilon
_{k}\searrow0$ be such that $b(\varepsilon_{k})\leq\dfrac{\varepsilon}{%
4^{k}~T}.$ Let $N_{k}=\left[ \dfrac{T}{\varepsilon_{k}}\right] $ and $t_{i}=%
\dfrac{(i-1)T}{N_{k}}.$

Applying Theorem Arzel\'{a}--Ascoli we see that the set%
\begin{equation*}
\begin{array}{l}
\mathcal{K}_{\varepsilon}=\Big\{z\in C([0,T];\mathbb{R}^{d}):\left\vert
z\left( 0\right) \right\vert \leq M,\smallskip \\
\quad\quad\quad\quad\quad\quad\quad\quad\quad\sup\limits_{1\leq i\leq
N_{k}}\sup\limits_{0<s\leq\varepsilon_{k}}\left\vert z\left( t_{i}+s\right)
-z\left( t_{i}\right) \right\vert \leq\gamma_{k},\forall k\in\mathbb{N}%
^{\ast}\Big\}%
\end{array}%
\end{equation*}
is compact in $C([0,T];\mathbb{R}^{d}).$

From Markov's inequality and $\left( jj\right) $
\begin{align*}
\mathbb{P}\left( X^{n}\notin\mathcal{K}_{\varepsilon}\right) & \leq\mathbb{P}%
(\{\left\vert X_{0}^{n}\right\vert >M\})+\sum\limits_{k\in \mathbb{N}%
^{\ast}}\sum\limits_{i=1}^{N_{k}}\mathbb{P}(\{\sup\limits_{0\leq
s\leq\varepsilon_{k}}\left\vert X_{t_{i}+s}^{n}-X_{t_{i}}^{n}\right\vert
>\gamma_{k}\}) \\
& <\frac{\varepsilon}{2}+\sum\limits_{k\in\mathbb{N}^{\ast}}\sum
\limits_{i=1}^{N_{k}}\dfrac{\varepsilon_{k}\times b(\varepsilon_{k})}{%
\gamma_{k}^{\alpha}}=\frac{\varepsilon}{2}+\sum\limits_{k\in\mathbb{N}^{\ast
}}\dfrac{N_{k}\times\varepsilon_{k}\times b(\varepsilon_{k})}{\gamma
_{k}^{\alpha}}\leq\varepsilon.
\end{align*}
The proof is complete now.\hfill$\smallskip$
\end{proof}

\begin{proposition}
\label{technical result 7}Let $g:\mathbb{R}_{+}\rightarrow\mathbb{R}_{+}$ be
a continuous function satisfying $g\left( 0\right) =0$ and $G:C\left(
\mathbb{R}_{+};\mathbb{R}^{d}\right) \rightarrow\mathbb{R}_{+}$ be a mapping
which is bounded on compact subsets of $C\left( \mathbb{R}_{+};\mathbb{R}%
^{d}\right) .$ Let $X^{n},Y^{n}$, $n\in\mathbb{N}^{\ast}$, be random
variables with values in $C\left( \mathbb{R}_{+};\mathbb{R}^{d}\right) .$ If
$\left\{ Y^{n}:n\in\mathbb{N}^{\ast}\right\} $ is tight and for all $n\in%
\mathbb{N}^{\ast}$%
\begin{equation*}
\begin{array}{rl}
\left( i\right) & \left\vert X_{0}^{n}\right\vert \leq G\left( Y^{n}\right)
,\;a.s.\smallskip \\
\left( ii\right) & \mathbf{m}_{X^{n}}\left( \varepsilon;\left[ 0,T\right]
\right) \leq G\left( Y^{n}\right) g\left( \mathbf{m}_{Y^{n}}\left(
\varepsilon;\left[ 0,T\right] \right) \right) ,\;a.s.,\;\;\forall
~\varepsilon,T>0,%
\end{array}%
\end{equation*}
then $\left\{ X^{n}:n\in\mathbb{N}^{\ast}\right\} $ is tight.
\end{proposition}

\begin{proof}
Let $\delta>0$ be arbitrary. Then there exists a compact set $K_{\delta
}\subset C\left( \left[ 0,\infty\right[ ;\mathbb{R}^{d}\right) $ such that
for all $n\in\mathbb{N}^{\ast}$%
\begin{equation*}
\mathbb{P}\left( Y^{n}\notin K_{\delta}\right) <\delta~.
\end{equation*}
Define $N_{\delta}=\sup\limits_{x\in K_{\delta}}G\left( x\right) .$ Then%
\begin{equation*}
\mathbb{P}\left( \left\vert X_{0}^{n}\right\vert >N_{\delta}\right) <\delta.
\end{equation*}
Let $a>0$ be arbitrary. There exists $\varepsilon_{0}>0$ such that
\begin{equation*}
\sup_{x\in K_{\delta}}\left[ g\left( \mathbf{m}_{x}\left( \varepsilon ;\left[
0,T\right] \right) \right) \right] <\frac{a}{N_{\delta}},\;\;\forall~0<%
\varepsilon<\varepsilon_{0}~.
\end{equation*}
Consequently for all $n\in\mathbb{N}^{\ast},$%
\begin{equation*}
\begin{array}{l}
\mathbb{P}\left( \mathbf{m}_{X^{n}}\left( \varepsilon;\left[ 0,T\right]
\right) \geq a\right) \leq\mathbb{P}\left[ g\left( \mathbf{m}_{Y^{n}}\left(
\varepsilon;\left[ 0,T\right] \right) \right) \geq\frac {a}{N_{\delta}},\
Y^{n}\in K_{\delta}\right] \smallskip \\
\quad\quad\quad\quad\quad\quad\quad\quad\quad\quad\quad+\mathbb{P}\left(
Y^{n}\notin K_{\delta}\right) \leq\delta%
\end{array}%
\end{equation*}
and the result follows.\hfill
\end{proof}

\subsection{It\^{o}'s stochastic integral}

In this subsection we consider $\left\{ B_{t}:t\geq0\right\} $ to be a $k$%
--dimensional Brownian motion on a stochastic basis (which is supposed to be
complete and right--continuous) $\left( \Omega,\mathcal{F},\mathbb{P},\{%
\mathcal{F}_{t}\}_{t\geq0}\right) $.

Let $S_{d}\left[ 0,T\right] $ be the space of p.m.c.s.p. $X:\Omega \times%
\left[ 0,T\right] \rightarrow\mathbb{R}^{d}$ and $\Lambda_{d}\left(
0,T\right) $ the space of p.m.c.s.p. $X:\Omega\times\left[ 0,T\right]
\rightarrow\mathbb{R}^{d}$ such that%
\begin{equation*}
\int_{0}^{T}\left\vert X_{t}\right\vert ^{2}dt<\infty,\;\mathbb{P}\text{%
--a.s.}
\end{equation*}
Write $S_{d}$ (and $\Lambda_{d}$) for space of p.m.c.s.p. $X:\Omega \times%
\left[ 0,T\right] \rightarrow\mathbb{R}^{d}$ such that the restriction of $X$
to $\left[ 0,T\right] $ belongs to $S_{d}$ (respectively to $\Lambda_{d}$).

If $X\in S_{d\times k}$ and $B$ is an $\mathbb{R}^{k}$--Brownian motion,
then the stochastic process $\left\{ \left( X_{t},B_{t}\right)
:t\geq0\right\} $ can be see as a random variable with values in the space $%
C(\mathbb{R}_{+},\mathbb{R}^{d\times k})\times C(\mathbb{R}_{+},\mathbb{R}%
^{k})$. The law of this random variable will be denoted $\mathcal{L}\left(
X,B\right) .$

\begin{proposition}[Corollary 2.13 in \protect\cite{pa-ra/14}]
\label{technical result 11}Let $X,\hat{X}\in S_{d}\left[ 0,T\right] $, $B,%
\hat{B}$ be two $\mathbb{R}^{k}$--Brownian motions and $g:\mathbb{R}%
_{+}\times \mathbb{R}^{d}\rightarrow \mathbb{R}^{d\times k}$ be a function
such that%
\begin{equation*}
\begin{array}{l}
g\left( \cdot ,y\right) \text{ is measurable }\forall ~y\in \mathbb{R}%
^{d},\smallskip \\
y\mapsto g\left( t,y\right) \text{ is continuous }dt-a.e.%
\end{array}%
\end{equation*}%
If%
\begin{equation*}
\mathcal{L}\left( X,B\right) =\mathcal{L}(\hat{X},\hat{B})\;\text{on }C(%
\mathbb{R}_{+},\mathbb{R}^{d+k}),
\end{equation*}%
then%
\begin{equation*}
\mathcal{L}\big(X,B,\int_{0}^{\cdot }g\left( s,X_{s}\right) dB_{s}\big)=%
\mathcal{L}\big(\hat{X},\hat{B},\int_{0}^{\cdot }g(s,\hat{X}_{s})d\hat{B}_{s}%
\big)\;\text{on }C(\mathbb{R}_{+},\mathbb{R}^{d+k+d}).
\end{equation*}
\end{proposition}

We present now a continuity property of the mapping%
\begin{equation*}
\left( X,B\right) \longrightarrow\int _{0}^{T}X_{s}dB_{s}~.
\end{equation*}
Given $B:\Omega\times\mathbb{R}_{+}\longrightarrow\mathbb{R}^{k}$ and $%
X:\Omega\times\mathbb{R}_{+}\longrightarrow\mathbb{R}^{d\times k}$ be two
stochastic processes, let $\mathcal{F}_{t}^{B,X}$ be the natural filtration
generated jointly by $B$ and $X$.

For the proof of the next Proposition see Proposition 2.4 in \cite{bu-ra/03}
or Proposition 2.14 in \cite{pa-ra/14}

\begin{proposition}
\label{technical result 8}Let $B$,$B^{n}$,$\tilde{B}^{n}:\Omega\times
\mathbb{R}_{+}\rightarrow\mathbb{R}^{k}$ and $X$,$X^{n}$,$\tilde{X}%
^{n}:\Omega\times\mathbb{R}_{+}\rightarrow\mathbb{R}^{d\times k}$ be
continuous stochastic processes such that

\begin{itemize}
\item[$\left( i\right) $] $\tilde{B}^{n}$ is $\mathcal{F}_{t}^{\tilde{B}^{n},%
\tilde{X}^{n}}-$Brownian motion $\forall~n\geq1;\smallskip$

\item[$\left( ii\right) $] $\mathcal{L}(\tilde{B}^{n},\tilde{X}^{n})=%
\mathcal{L}\left( B^{n},X^{n}\right) $ on $C(\mathbb{R}_{+},\mathbb{R}%
^{k}\times\mathbb{R}^{d\times k})$, for all $n\geq1;\smallskip$

\item[$\left( iii\right) $] $\int _{0}^{T}\left\vert
X_{s}^{n}-X_{s}\right\vert ^{2}ds+\sup\limits_{t\in\left[ 0,T\right]
}\left\vert B_{t}^{n}-B_{t}\right\vert \rightarrow0,$ in probability$%
\smallskip$, as $n\rightarrow\infty,$ for all $T>0.$
\end{itemize}

Then $\big(B^{n},\{\mathcal{F}_{t}^{B^{n},X^{n}}\}\big),n\geq1,\;$and $\big(%
B,\{\mathcal{F}_{t}^{B,X}\}\big)$ are Brownian motions and as $%
n\rightarrow\infty$%
\begin{equation}
\sup_{t\in\left[ 0,T\right] }\left\vert
\int_{0}^{t}X_{s}^{n}dB_{s}^{n}\longrightarrow\int_{0}^{t}X_{s}dB_{s}\right%
\vert \longrightarrow 0\quad\text{in probability}.  \label{technical ineq 3}
\end{equation}
\end{proposition}

\subsection{A forward stochastic inequality}

Let $X$, $\hat{X}\in S_{d}$ be two semimartingales defined by%
\begin{equation}
X_{t}=X_{0}+K_{t}+\displaystyle\int_{0}^{t}G_{s}dB_{s}~,\;t\geq0,\quad\hat {X%
}_{t}=\hat{X}_{0}+\hat{K}_{t}+\displaystyle\int_{0}^{t}\hat{G}%
_{s}dB_{s}~,\;t\geq0,  \label{semimart}
\end{equation}
where

\noindent$\lozenge$\ $K,\hat{K}\in S_{d}~;$

\noindent$\lozenge$\ $K_{\cdot}\left( \omega\right) $, $\hat{K}_{\cdot
}\left( \omega\right) \in BV_{loc}\left( \mathbb{R}_{+};\mathbb{R}%
^{d}\right) $, $K_{0}\left( \omega\right) =\hat{K}_{0}\left( \omega\right)
=0 $,\ $\mathbb{P}$--a.s.$~;$

\noindent$\lozenge$\ $G,\hat{G}\in\Lambda_{d\times k}~.$

Assume that there exist $p\geq1$ and $\lambda\geq0$\ and $V$ a bounded
variation p.m.c.s.p., with $V_{0}=0$, such that, as measures on $\mathbb{R}%
_{+},$%
\begin{equation}
\big\langle X_{t}-\hat{X}_{t},dK_{t}-d\hat{K}_{t}\big\rangle+\big(\dfrac{1}{2%
}m_{p}+9p\lambda\big)|G_{t}-\hat{G}_{t}|^{2}dt\leq|X_{t}-\hat{X}%
_{t}|^{2}dV_{t}.  \label{technical ineq 7}
\end{equation}

\begin{theorem}[Corollary 6.74 in \protect\cite{pa-ra/14}]
Let $p\geq 1$. If the assumption (\ref{technical ineq 7}) is satisfied with $%
\lambda >1$, then there exists a positive constant $C_{p,\lambda }$ such
that for all $\delta \geq 0$, $0\leq t\leq s:$%
\begin{equation*}
\mathbb{E}^{\mathcal{F}_{t}}\frac{||e^{-V}(X-\hat{X})||_{\left[ t,s\right]
}^{p}}{\left( 1+\delta ||e^{-V}(X-\hat{X})||_{\left[ t,s\right] }^{2}\right)
^{p/2}}\leq C_{p,\lambda }\displaystyle\frac{e^{-pV_{t}}|X_{t}-\hat{X}%
_{t}|^{p}}{\big(1+\delta e^{-2V_{t}}|X_{t}-\hat{X}_{t}|^{2}\big)^{p/2}}%
~,\quad \mathbb{P}\text{--a.s.}
\end{equation*}%
In particular for $\delta =0$%
\begin{equation*}
\mathbb{E}^{\mathcal{F}_{t}}||e^{-V}(X-\hat{X})||_{\left[ t,s\right]
}^{p}\leq C_{p,\lambda }~e^{-pV_{t}}|X_{t}-\hat{X}_{t}|^{p},\quad \mathbb{P}%
\text{--a.s.},
\end{equation*}%
for all $0\leq t\leq s.$
\end{theorem}

As a consequence of the above theorem, since%
\begin{equation*}
\frac{1}{2}\left( 1\wedge r\right) \leq\frac{r}{\left( 1+r^{2}\right) ^{1/2}}%
\leq1\wedge r,\quad\forall r\geq0,
\end{equation*}
we obtain:

\begin{corollary}
\label{technical result 5}If assumption (\ref{technical ineq 7}) is
satisfied with $\lambda>1$ and $p\geq1,$ then there exists a positive
constant $C_{p,\lambda}$ depending only on $\left( p,\lambda\right) $ such
that $\mathbb{P}-a.s.$%
\begin{equation*}
\mathbb{E}^{\mathcal{F}_{t}}\left[ 1\wedge||e^{-V}(X-\hat{X})||_{\left[ t,s%
\right] }^{p}\right] \leq C_{p,\lambda}\left[ 1\wedge|e^{-V_{t}}(X_{t}-\hat{X%
}_{t})|^{p}\right] ,
\end{equation*}
for all $0\leq t\leq s.$
\end{corollary}

\end{document}